\documentclass[letterpaper, 10 pt]{IEEEtran}  

\IEEEoverridecommandlockouts                              

\usepackage{mathtools}
\DeclarePairedDelimiter{\ceil}{\lceil}{\rceil}

\usepackage{blkarray}
\usepackage{amsthm}
\usepackage{amssymb}
\usepackage{amsmath}
\usepackage{multirow}

\usepackage{graphicx,epsfig,color,amsfonts,subfigure}
\usepackage{version,xspace}
\usepackage{tikz}
\usepackage{centernot}
\usepackage{url,doi}
\usepackage{environ}
\usepackage{bbm}
\usepackage{bbold}
\usepackage{bm}

\usepackage{algorithm}
\usepackage{algorithmic}
\usepackage{blkarray}
\usetikzlibrary{automata,positioning}

\usepackage{makecell}

\DeclareMathOperator*{\argmax}{argmax}
\DeclareMathOperator*{\argmin}{argmin}

\usetikzlibrary{calc}
\newcommand{\algostep}[1]{{\small\texttt{#1:}}\xspace}
\def\prob{\mathbb{P}}

\def\real{\mathbb{R}}

\def\vspecdots{\vbox{\baselineskip=2pt \lineskiplimit=0pt
    \kern2pt \hbox{.}\hbox{.}\hbox{.}}}
\newcommand{\until}[1]{\{1,\dots, #1\}}

\newcommand{\subscr}[2]{#1_{\textup{#2}}}

\newcommand{\map}[3]{#1: #2 \rightarrow #3}

\newcommand{\E}{\mathbb{E}}

\newcommand{\Hrt}{\mathbb{J}}
\newcommand{\Hrate}{\subscr{\mathbb{H}}{rate}}

\renewcommand{\theenumi}{(\roman{enumi})}

\DeclareMathOperator{\vecz}{vec}

\newcommand\oprocendsymbol{\hbox{$\square$}}
\newcommand\oprocend{\relax\ifmmode\else\unskip\hfill\fi\oprocendsymbol}

\renewcommand{\baselinestretch}{0.9815}

\DeclareSymbolFont{bbold}{U}{bbold}{m}{n}
\DeclareSymbolFontAlphabet{\mathbbold}{bbold}
\newcommand{\vect}[1]{\mathbbold{#1}}
\newcommand{\vectorones}[1][]{\vect{1}_{#1}}

\newtheorem{theorem}{Theorem}
\newtheorem{lemma}[theorem]{Lemma}

\newtheorem{remark}[theorem]{Remark}
\newtheorem{example}[theorem]{Example}
\newtheorem{examples}[theorem]{Examples}
\newtheorem{definition}[theorem]{Definition}
\newtheorem{problem}{Problem}

\title{\LARGE \bf
Markov Chains with Maximum Return Time Entropy \\for Robotic Surveillance
}
\author{Xiaoming Duan,~\IEEEmembership{Student member,~IEEE}, Mishel George, Francesco Bullo,~\IEEEmembership{Fellow,~IEEE}
\thanks{This work has been supported in part by Air Force Office of Scientific
Research award FA9550-15-1-0138.}
\thanks{Xiaoming Duan, Mishel George, and Francesco Bullo are with the Mechanical
Engineering Department and the Center of Control, Dynamical
Systems and Computation, UC Santa Barbara, CA 93106-5070, USA.
{\tt\small \{xmduan,mishel,bullo\}@engineering.ucsb.edu}}%
}

\begin{document}

\maketitle
\thispagestyle{empty}
\pagestyle{empty}

\begin{abstract}
  Motivated by robotic surveillance applications, this paper studies the
  novel problem of maximizing the return time entropy of a Markov chain,
  subject to a graph topology with travel times and stationary
  distribution.  The return time entropy is the weighted average, over all
  graph nodes, of the entropy of the first return times of the Markov
  chain; this objective function is a function series that does not admit
  in general a closed form.

  The paper features theoretical and computational contributions.  First,
  we obtain a discrete-time delayed linear system for the return time
  probability distribution and establish its convergence properties. We
  show that the objective function is continuous over a compact set and
  therefore admits a global maximum; a unique globally-optimal solution is
  known only for complete graphs with unitary travel times.  We then
  establish upper and lower bounds between the return time entropy and the
  well-known entropy rate of the Markov chain.  To compute the optimal
  Markov chain numerically, we establish the asymptotic equality between
  entropy, conditional entropy and truncated entropy, and propose an
  iteration to compute the gradient of the truncated entropy. Finally, we
  apply these results to the robotic surveillance problem. Our numerical
  results show that, for a model of rational intruder over prototypical
  graph topologies and test cases, the maximum return time entropy chain
  performs better than several existing Markov chains.
\end{abstract}

\section{Introduction}
\paragraph*{Problem description and motivation}
Given a Markov chain, the first return time of a given node is the first
time that the random walker returns to the starting node; this is a
discrete random variable with infinite support and whose randomness is
measured by its entropy. In this paper, given a strongly connected directed
graph with integer-valued travel times (weights) and a prescribed
stationary distribution, we study Markov chains with maximum return time
entropy. Here the return time entropy of a Markov chain is a weighted
average of the entropy of different states' return times with weights equal
to the stationary distribution.

This optimization problem is motivated by robotic applications.  We design
stochastic surveillance strategies with an entropy maximization objective
in order to thwart intruders who plan their attacks based on observations
of the surveillance agent. The randomness in the first return time is
desirable because an intelligent intruder observing the inter-visit times
of the surveillance agent is confronted with a maximally unpredictable
return pattern by the surveillance agent.

\paragraph*{Literature review}
Ekroot \emph{et al.} studied the entropy of Markov trajectories in
\cite{LE-TMC:93}, i.e., the entropy of paths with specified initial and
final states. The authors establish an equivalence relationship between the
entropy of return Markov trajectories (paths with the same initial and
final states) and the entropy rate of the Markov chains.  Compared
with~\cite{LE-TMC:93}, we study here the return time random variable, by
lumping return trajectories with the same length.  Importantly, our
formulation incorporates travel times, as motivated by robotic
applications.

The problem of designing robotic surveillance strategies has been widely
studied \cite{NA-SK-GAK:08, SA-EF-SLS:14,
  HX-BF-FF-BD-AP-MT-MD-FW-AR-MN-JM:17,JY-SK-DR:15}. Stochastic surveillance
strategies, which emphasize the unpredictability of the movement of the
patroller, are desirable since they are capable of defending against
intelligent intruders who aim to avoid detection/capture. One of the main
approaches to the design of robotic stochastic surveillance strategies is
to adopt Markov chains; e.g., see the early reference~\cite{JG-JB:05} and
the more recent~\cite{BA-SDB:15,  SB-SJC-FYH:17,GC-AS:11,
  NN-AR-JVH-VI:16}. Srivastava \emph{et al.}~\cite{KS-DMS-MWS:09}
justified the Markov chain-based stochastic surveillance strategy by
showing that for the deterministic strategies, in addition to
predictability, it is also hard to specify the visit frequency. However,
for the finite state irreducible Markov chains, the visit frequency is
embedded naturally in the stationary distribution. Patel \emph{et al.}
\cite{RP-PA-FB:14b} studied the Markov chains with minimum weighted mean
hitting time where weights are travel times on edges. For the class of
reversible Markov chains, they formulated the problem as a convex
optimization problem. An extension of the mean hitting time to the
multi-agent case was studied in \cite{RP-AC-FB:14k}. Asghar \emph{et al.}
\cite{ABA-SLS:16} introduced different intruder models and designed a
pattern search-based algorithm to solve for a Markov chain that minimizes
the expected reward of the intruders. Recently, George \emph{et al.}
\cite{MG-SJ-FB:17b} studied and quantified the unpredictability of the
Markov chains and designed the maxentropic surveillance strategies by
maximizing the entropy rate of Markov chains \cite{EA-NCM:14,
  YC-TG-MP-AT:17a}. Compared with \cite{MG-SJ-FB:17b}, our problem
formulation features a new notion of entropy, a directed graph topology,
and travel times; these three features render the results potentially more
widely applicable and more relevant (see also the performance comparison
among multiple Markov chains later in the paper).

\paragraph*{Contributions}
In this paper, we propose a new metric that measures the unpredictability
of the Markov chains over a directed graph with travel times. This novel
formulation is of interest in the general study of Markov chains as well as
for its applications to robotic surveillance. The main contributions of
this paper are sixfold. First, we introduce and analyze a discrete-time
delayed linear system for the return time probabilities of the Markov
chains. This system incorporates integer-valued travel times on the
directed graph. Second, we propose to characterize the unpredictability of
a Markov chain by the return time entropy and formulate an entropy
maximization problem. Third, we prove the well-posedness of the return time
entropy maximization problem, i.e., the objective function is continuous
over a compact set and thus admits a global maximum. For the case of
unitary travel times, we derive an upper bound for the return time entropy
and solve the problem analytically for the complete graph. Fourth, we
compare the return time entropy with the entropy rate of Markov chains;
specifically, we prove that the return time entropy is lower bounded by the
entropy rate and upper bounded by the number of nodes times of the entropy
rate. Fifth, in order to compute Markov chains with maximum return time
entropy numerically, we truncate the return time entropy and show that the
truncated entropy is asymptotically equivalent to both the original
objective and the practically useful conditional return time entropy. We
also characterize the gradient of the truncated return time entropy and use
it to implement a gradient projection method. Sixth, we apply our solution
to different prototypical robotic surveillance scenarios and test cases and
show that, for a model of rational intruder, the Markov chain with maximum
return time entropy outperforms several existing Markov chains.

\paragraph*{Paper organization}
This paper is organized as follows. We formulate the return time entropy
maximization problem in Section~\ref{sec:ProblemFormulation}. We establish
the properties of the return time entropy in
Section~\ref{sec:properties}. The approximation analysis and the gradient
formulas are provided in Section~\ref{sec:approximation}. We present the
simulation results regarding the robotic surveillance problem in
Section~\ref{sec:Simulation}. Section~\ref{sec:Conclusion} concludes the
paper.

\subsection*{Notation and useful lemmas}
Let $\real$, $\mathbb{Z}_{\geq0}$, and $\mathbb{Z}_{>0}$ denote the set of
real numbers, nonnegative and positive integers, respectively. Let
$\mathbb{1}_n$ and $\mathbb{0}_n$ denote column vectors in $\real^n$ with
all entries being $1$ and $0$. $I_n\in \real^{n\times n}$ is the identity
matrix. $\mathbb{e}_i$ denotes the $i$-th vector in the standard basis,
whose dimension will be made clear when it appears. $[S]$ denotes a
diagonal matrix with diagonal elements being $S$ if $S$ is a vector, or
being the diagonal of $S$ if $S$ is a square matrix. Let $\otimes$ denote
the Kronecker product. $\vecz(\cdot)$ is the vectorization operator that
converts a matrix into a column vector. The following lemmas are useful.


\begin{lemma}(A uniform bound for stable matrices \cite[Proposition D.3.1]{MHAD-RBV:85})\label{lemma:SolutionBound}
  Assume the matrix subset $\mathcal{A}\subset\real^{n\times{n}}$ is
  compact and satisfies
  \begin{equation}\nonumber
    \rho_{\mathcal{A}} := \max_{A\in\mathcal{A}}\rho(A)<1.
  \end{equation}
  Then for any $\lambda\in(\rho_{\mathcal{A}},1)$ and for any induced
  matrix norm $\|\cdot\|$, there exists $c>0$ such that
  \begin{equation}\nonumber
    \|A^k\| \leq c\lambda^k, \quad \text{for all } A\in\mathcal{A}\text{ and } k\in\mathbb{Z}_{\geq0} .
  \end{equation}
\end{lemma}

\begin{lemma}(Weierstrass M-test \cite[Theorem 7.10]{WR:76})\label{lemma:Weierstrass} Given a set
  $\mathcal{X}$, consider the sequence of functions
  $\{\map{f_k}{\mathcal{X}}{\real}\}_{k\in\mathbb{Z}_{>0}}$. If there
  exists a sequence of scalars $\{M_k\in\real\}_{k\in\mathbb{Z}_{>0}}$ satisfying
  $\sum_{k=1}^\infty M_k<\infty$ and
  \begin{equation}\nonumber
    |f_k(x)|\leq M_k, \quad\text{for all } x\in \mathcal{X},k\in \mathbb{Z}_{>0},
  \end{equation}
  then $\sum_{k=1}^\infty f_k$ converges uniformly on $\mathcal{X}$.
\end{lemma}

\begin{lemma}(Geometric distribution generates maximum entropy \cite{SG-AS:85})\label{lemma:MaximumEntropy}
  Given a discrete random variable $Y \in \mathbb{Z}_{>0}$ and
  $\mathbb{E}[Y] = \mu\geq 1$, the probability distribution with
  maximum entropy is
  \begin{equation}\nonumber
    \mathbb{P}[Y=k] = (1-\frac{1}{\mu})^{k-1} \frac{1}{\mu},\quad k\in\mathbb{Z}_{>0},
  \end{equation}
  with entropy
  \begin{equation}\label{eq:maximumentropy}
    \mathbb{H}(Y) = \mu \log \mu -(\mu-1)\log(\mu-1).
  \end{equation}
\end{lemma}

\section{Problem formulation}\label{sec:ProblemFormulation}
We start by reviewing the basics of discrete-time Markov chains.  A
finite-state discrete-time Markov chain with state space $\until{n}$ is a
sequence of random variables taking values in $\until{n}$ and satisfying
the Markov property. Let $X_k$ be the random variable at time
$k\in\mathbb{Z}_{\geq0}$, then a time-homogeneous Markov chain satisfies,
for all $i,j\in\until{n}$ and $k\in\mathbb{Z}_{\geq0}$, $
\mathbb{P}(X_{k+1}=j\,|\,X_{k}=i,\dots,X_{1}=i_{1},X_{0}=i_0) =\mathbb{P}(X_{k+1}=j\,|\,X_{k}=i)=p_{ij},
$
where $p_{ij}$ is the transition probability from state $i$ to state $j$
and $P=\{p_{ij}\}\in\real^{n\times{n}}$ is the transition matrix satisfying
$P\geq 0$ and $P\mathbb{1}_n=\mathbb{1}_n$; see \cite{JGK-JLS:76},
\cite{JRN:97}. A probability distribution $\bm{\pi}\in\real^n$ is
\emph{stationary} for the Markov chain with transition matrix $P$ if it
satisfies $\bm{\pi}\geq0$, $\bm{\pi}^\top\mathbb{1}_n=1$ and
$\bm{\pi}^\top=\bm{\pi}^\top P$. A Markov chain is \emph{irreducible} if
its transition diagram is a strongly connected graph. A Markov chain that
satisfies the detailed balance equation $[\bm{\pi}]P=P^\top[\bm{\pi}]$ is
\emph{reversible}. A discrete-time Markov chain is also referred to as a random walk on a
graph.

\subsection{Return time of random walks}\label{sec:returntime}

In this paper, we consider a strongly connected directed weighted graph
$\mathcal{G}=\{V,\mathcal{E},W\}$, where $V$ denotes the set of $n$ nodes
$\until{n}$, $\mathcal{E}\subset V\times V$ denotes the set of edges, and
$W\in\mathbb{Z}_{\geq0}^{n\times n}$ is the integer-valued weight (travel
time) matrix with $w_{ij}$ being the one-hop travel time from node $i$ to
node $j$. If $(i,j)\notin\mathcal{E}$, then $w_{ij}=0$; if
$(i,j)\in\mathcal{E}$, then $w_{ij}\geq 1$. Let
$w_{\max}=\max_{i,j}\{w_{ij}\}$ be the maximum travel time.

Given the graph $\mathcal{G}=\{V,\mathcal{E},W\}$, let $X_k
\in \until{n}$ denote the location of a random walk on $\mathcal{G}$ following a
transition matrix $P$ at time $k \in \mathbb{Z}_{\geq0}$. For any
pair of nodes $i,j\in V$, the \emph{first hitting time} from $i$ to $j$,
denoted by $T_{ij}$, is the first time the random walk reaches node $j$
starting from node $i$, that is
\begin{equation}\label{eq:passagetime}
T_{ij}=\min\Big\{\sum_{k'=0}^{k-1}w_{X_{k'}X_{k'+1}}\,|\,X_0=i,X_k=j,k\geq1\Big\}.
\end{equation}
In particular, the \emph{return time} $T_{ii}$ of node $i$ is the first
time the random walk returns to node $i$ starting from node $i$. Let the
$(i,j)$-th element of the \emph{first hitting time probability matrix}
$F_k$ denote the probability that the random walk reaches node $j$ for
the first time in exactly $k$ time units starting from node $i$, i.e.,
$F_k(i,j)=\mathbb{P}(T_{ij}=k)$.

\subsection{Return time entropy of random walks}
For an irreducible Markov chain, the return time $T_{ii}$ of each state $i$
is a well-defined random variable over $\mathbb{Z}_{>0}$. We define the
\emph{return time entropy of state $i$} by
\begin{align}
  \mathbb{H}(T_{ii})&=-\sum\limits_{k=1}^\infty \mathbb{P}(T_{ii}=k)\log \mathbb{P}(T_{ii}=k) \nonumber \\
  &=-\sum\limits_{k=1}^\infty F_k(i,i)\log F_k(i,i), \label{eq:ret-time-entropy}
\end{align}
where the logarithm is the natural logarithm and $0\log0=0$.

\begin{remark}(Coprime travel times)
  The return time entropy of states does not change when we scale the
  travel times on all edges simultaneously by the same factor. Therefore,
  we assume the weights on the graph are coprime.
\end{remark}

\begin{definition}(The set of Markov chains $\epsilon$-conforming to a graph)\label{def:MCconforming}
  Given a strongly connected directed weighted graph $\mathcal{G}=\{V,\mathcal{E},W\}$
  with $n$ nodes and the stationary distribution ${\bm{\pi}}>0$, pick a
  minimum edge weight $\epsilon>0$, the set of Markov chains
  \emph{$\epsilon$-conforming} to $\mathcal{G}$ is defined by
  \begin{align*}
    \mathcal{P}_{\mathcal{G},\bm{\pi}}^\epsilon=\{P\in \real^{n\times n}\,|\,&p_{ij}\geq \epsilon \text{ if } (i,j)\in \mathcal{E},\\
    &p_{ij}=0 \text{ if } (i,j)\notin \mathcal{E},\\
    &P\mathbb{1}_n=\mathbb{1}_n,{\bm{\pi}}^\top P={\bm{\pi}}^\top\}.
  \end{align*}
\end{definition}

\begin{definition}(Return time entropy)
Given a set $\mathcal{P}_{\mathcal{G},\bm{\pi}}^\epsilon$, define the
\emph{return time entropy} function
$\Hrt:\mathcal{P}_{\mathcal{G},\bm{\pi}}^\epsilon
\mapsto \real_{\geq 0}$ by
\begin{equation}\label{eq:weightedentropy}
\Hrt(P)=\sum_{i=1}^n\pi_i\mathbb{H}(T_{ii}).
\end{equation}
\end{definition}

\begin{remark}(The expectation of the first return time)\label{rmk:expectedreturntime}
  For an irreducible Markov chain defined over a weighted graph with travel
  times, \cite[Theorem~6]{RP-PA-FB:14b} states
  \begin{equation}\label{eq:ETii}
    \mathbb{E}[T_{ii}]=\frac{\bm{\pi}^\top(P\circ W)\mathbb{1}_n}{\pi_i},
  \end{equation}
  where $\circ$ is the Hadamard element-wise product.  For unitary travel
  times, this formula reduces to the usual $\E[T_{ii}]=1/\pi_i$.  In both
  cases, the first return times expectations are inversely proportional to
  the entries of~$\bm{\pi}$.
\end{remark}



In general, it is difficult to obtain the closed-form expression for the
return time entropy function.

\begin{examples}(Two special cases with unitary travel times)\label{example:closedform}
  The elementary proofs of the following results are omitted in the
  interest of brevity.
  \begin{enumerate}
  \item (Two-node complete graph case) Given a two-node complete graph
    $\mathcal{G}$ with unit weights, if the transition matrix
    $P\in\mathcal{P}_{\mathcal{G},\bm{\pi}}^\epsilon$ has the following
    form
    \begin{equation*}
      P =\begin{bmatrix}
      p_{11}&p_{12}\\
      p_{21}&p_{22}
      \end{bmatrix},
    \end{equation*}
    then the return time entropy function is
    \begin{align*}
      \Hrt(P)&=-2\pi_1p_{11}\log(p_{11})-2\pi_2p_{22}\log(p_{22})\\
      &\quad-2\pi_1p_{12}\log(p_{12})-2\pi_2p_{21}\log(p_{21}).
    \end{align*}

  \item (Complete graph case with special structure) Given an $n\geq2$-node
    complete graph $\mathcal{G}$ with unit weights and the stationary distribution
    $\bm{\pi}=\frac{1}{n}\mathbb{1}_n$, if the transition matrix
    $P\in\mathcal{P}_{\mathcal{G},\bm{\pi}}^\epsilon$ has the form
    \begin{equation*}
      P=(a-b)I_n+b\mathbb{1}_n\mathbb{1}_n^\top,
    \end{equation*}
    for any $a\geq 0$ and $b>0$ satisfying $a+(n-1)b=1$, then the return time entropy function is
    \begin{align*}
      \Hrt(P)&=-a\log(a)-(n-1)b\log\big((n-1)b^2\big)\\
      &\quad -(n-1)(1-b)\log(1-b).
    \end{align*}
  \end{enumerate}
\end{examples}

In this paper, we are interested in the following problem.

\begin{problem}\label{prob:weightedentropy}
(Maximization of the return time entropy) Given a strongly connected directed weighted graph $\mathcal{G}=\{V,\mathcal{E},W\}$ and the stationary distribution $\bm{\pi}>0$, pick a minimum edge weight $\epsilon>0$, the maximization of the return time entropy is as follows.
\begin{align*}
& \text{maximize}
& &\Hrt(P)\\
& \text{subject to}
&& P\in\mathcal{P}_{\mathcal{G},\bm{\pi}}^\epsilon
\end{align*}
\end{problem}

\section{Properties of the return time entropy}\label{sec:properties}
\subsection{Dynamical model for hitting time probabilities}
In this subsection, we characterize a dynamical model for the first hitting
time probabilities and establish several important properties of the model.

\begin{theorem}(Linear dynamics for the first hitting time probabilities)\label{thm:linear-dynamics}
  Consider a transition matrix $P\in\mathbb{R}^{n\times n}$ that is
  nonnegative, row-stochastic and irreducible. Then
  \begin{enumerate}
  \item the hitting time probabilities $F_k$, $k\in\mathbb{Z}_{>0}$,
    satisfy the discrete-time delayed linear system with a finite number of
    impulse inputs:
    \begin{multline}
      \vecz(F_k)   =\vecz(P\circ \mathbf{1}_{\{k\mathbb{1}_n\mathbb{1}_n^\top=W\}}) \\
      +\sum_{i=1}^n\sum_{j=1}^np_{ij}(E_j\otimes\mathbb{e}_i\mathbb{e}_j^\top)\vecz(F_{k-w_{ij}}),\label{eq:dynamicswithtravel}
    \end{multline}
    where $E_{i}=[\mathbb{1}_n-\mathbb{e}_i]\in\mathbb{R}^{n\times n}$, and the
    initial conditions are $\vecz(F_{k})=\mathbb{0}_{n^2}$ for all $k\leq0$;

  \item if the weights are unitary, i.e., $w_{ij}\in\{0,1\}$, then the
    hitting time probabilities satisfy
    \begin{equation}\label{eq:RecursiveInVec}
      \vecz(F_k) =(I_n\otimes P)(I_{n^2} - [\vecz(I_n)])\vecz(F_{k-1}),
    \end{equation}
    where the initial condition is $F_1 = P$.
  \end{enumerate}
\end{theorem}
\begin{proof}
By definition in \eqref{eq:passagetime}, $F_k(i,j)$ satisfies
the following recursive formula for $k\in \mathbb{Z}_{>0}$
\begin{equation}\label{eq:RecursiveOri}
    F_k(i,j) = p_{ij}\mathbf{1}_{\{k=w_{ij}\}}+\sum\limits_{h=1,h\neq j}^np_{ih}F_{k-w_{ih}}(h,j),
\end{equation}
where $\mathbf{1}_{\{\cdot\}}$ is the indicator function and $F_k(i,j)=0$ for all $k\leq0$ and $i,j\in V$.

Let $D_k(i)\in\mathbb{R}^{n\times n}$ be a matrix associated with node $i$ at time $k$ that has the form
\begin{equation*}
D_k(i)=\sum_{j\in\mathcal{N}_i}\mathbb{e}_j\mathbb{e}_j^\top F_{k-w_{ij}},
\end{equation*}
where $\mathcal{N}_i$ is the set of out-going neighbors of node $i$. Then, \eqref{eq:RecursiveOri} can be written in the following matrix form
\begin{equation}\label{eq:Fkmatrix}
F_k=P\circ \mathbf{1}_{\{k\mathbb{1}_n\mathbb{1}_n^\top=W\}} + \sum_{i=1}^n\mathbb{e}_i\mathbb{e}_i^\top P (D_k(i)-[D_k(i)]).
\end{equation}
Vectorizing both sides of \eqref{eq:Fkmatrix}, we have
\begin{multline*}
    \vecz(F_k)   =\vecz(P\circ \mathbf{1}_{\{k\mathbb{1}_n\mathbb{1}_n^\top=W\}}) \\
    +\sum_{i=1}^n(I_n\otimes\mathbb{e}_i\mathbb{e}_i^\top P)(I_{n^2}-[\vecz(I_n)])\vecz(D_k(i)).
\end{multline*}
Note that
\begin{equation*}
\vecz(D_k(i)) = \sum_{j\in\mathcal{N}_i}(I_n\otimes\mathbb{e}_j\mathbb{e}_j^\top)\vecz(F_{k-w_{ij}}),
\end{equation*}
and
\begin{equation*}
(I_{n^2}-[\vecz(I_n)])(I_n\otimes\mathbb{e}_j\mathbb{e}_j^\top)=E_j\otimes\mathbb{e}_j\mathbb{e}_j^\top.
\end{equation*}
Therefore, we have \eqref{eq:dynamicswithtravel}.

Moreover, if the travel times are unitary, then $F_1=P$ and
\begin{equation}\label{eq:sumindividual}
\sum_{i=1}^n\sum_{j=1}^np_{ij}(E_j\otimes\mathbb{e}_i\mathbb{e}_j^\top)=(I_n\otimes P)(I_{n^2} - [\vecz(I_n)]).
\end{equation}
Thus, equation \eqref{eq:RecursiveInVec} follows.
\end{proof}

The dynamical system \eqref{eq:dynamicswithtravel} can be transformed to an
equivalent homogeneous linear system by restarting the system at $k=w_{M}$
with same system matrices and appropriate initial conditions. Moreover, we
can augment the system and obtain a discrete-time linear system without
delays. This equivalent augmented system is useful for example in studying
stability properties. For $k\geq1$, we have
\begin{equation}\label{eq:augmentedsystem}
\begin{bmatrix}
\vecz(F_{k+w_{\max}})\\
\vecz(F_{k+w_{\max}-1})\\
\vdots\\
\vecz(F_{k+1})\\
\end{bmatrix}=\Psi\begin{bmatrix}
\vecz(F_{k+w_{\max}-1})\\
\vecz(F_{k+w_{\max}-2})\\
\vdots\\
\vecz(F_{k})
\end{bmatrix},
\end{equation}
where
\begin{equation}\label{eq:sysmatrix}
\Psi=\begin{bmatrix}
\Phi_1&\Phi_2&\cdots&\cdots&\Phi_{w_{\max}}\\
I_{n^2}&\mathbb{0}_{n^2\times n^2}&\cdots&\cdots&\mathbb{0}_{n^2\times n^2}\\
\mathbb{0}_{n^2\times n^2}&I_{n^2}&\cdots&\cdots&\mathbb{0}_{n^2\times n^2}\\
\vdots&\vdots&\ddots&\cdots&\mathbb{0}_{n^2\times n^2}\\
\mathbb{0}_{n^2\times n^2}&\cdots&\cdots&I_{n^2}&\mathbb{0}_{n^2\times n^2}\\
\end{bmatrix},
\end{equation}
and for $h\in[1,w_{\max}]$,
\begin{equation}\label{eq:phimatrix}
\Phi_h=\sum_{i=1}^n\sum_{j=1}^np_{ij}(E_j\otimes\mathbb{e}_i\mathbb{e}_j^\top)\mathbf{1}_{\{w_{ij}=h\}}.
\end{equation}
The initial conditions for \eqref{eq:augmentedsystem} can be computed using \eqref{eq:dynamicswithtravel}. For brevity, we denote $
\begin{bmatrix}
\vecz(F_{k+w_{\max}-1})&
\cdots&
\vecz(F_{k})
\end{bmatrix}^\top$ by $\vecz(\tilde{F}_k)^\top$.


\begin{lemma}(Properties of the linear dynamics for the first hitting time probabilities) \label{lemma:systempropertiestravel}
If $P\in \real^{n\times n}$ is nonnegative, row-stochastic and
irreducible, then
\begin{enumerate}
\item the matrix $(I_n\otimes P)(I_{n^2} - [\vecz(I_n)])$ is row-substochastic with $\rho\big((I_n\otimes P)(I_{n^2} - [\vecz(I_n)])\big)<1$.
\item the delayed discrete-time linear system with a finite number of impulse inputs \eqref{eq:dynamicswithtravel} is asymptotically stable;
\item $\vecz(F_k)\geq 0$ for $k\in \mathbb{Z}_{>0}$ and $\sum_{k=1}^{\infty}\vecz(F_k) = \mathbb{1}_{n^2\times 1}$.
\end{enumerate}
\end{lemma}
\begin{proof}
Regarding (i), note that the matrix $(I_n\otimes P)(I_{n^2} - [\vecz(I_n)])$ is block diagonal with
the $i$-th block being $PE_{i}$. Since $P$ is
irreducible, there is at least one positive entry in each column of
$P$. Therefore $PE_{i}$'s are
row-substochastic and so is $(I_n\otimes P)(I_{n^2} - [\vecz(I_n)])$. By \cite[Lemma~2.2]{RP-AC-FB:14k}, $\rho(PE_{i})<1$ for all $i\in\until{n}$ and $\rho((I_n\otimes P)(I_{n^2} - [\vecz(I_n)]))=\max_{i}\rho(PE_{i})<1$.

Regarding (ii), since we can rewrite \eqref{eq:dynamicswithtravel} as \eqref{eq:augmentedsystem} with appropriate initial conditions and $\Phi_i$'s are nonnegative, by the stability criterion for delayed linear systems \cite[Theorem 1]{AH-ET:98}, \eqref{eq:dynamicswithtravel} is asymptotically stable if
\[
\rho\Big(\sum_{i=1}^{w_{\max}}\Phi_i\Big)=\rho\big((I_n\otimes P)(I_{n^2} - [\vecz(I_n)])\big)<1,
\]
which is true by (i).

Regarding (iii), first note that all the system matrices are nonnegative, thus $\vecz(F_k)\geq 0$ for all $k\in \mathbb{Z}_{>0}$. Moreover, due to (ii), the delayed linear system \eqref{eq:dynamicswithtravel} is asymptotically stable. Summing both sides of \eqref{eq:dynamicswithtravel} over $k$, we have
\begin{small}
\begin{align*}
\sum_{k=1}^{\infty}\vecz(F_k)&=\vecz(P) + \sum_{i=1}^n\sum_{j=1}^np_{ij}(E_j\otimes\mathbb{e}_i\mathbb{e}_j^\top)\sum_{k=1}^{\infty}\vecz(F_k)\\
&=\vecz(P) + (I_n\otimes P)(I_{n^2} - [\vecz(I_n)])\sum_{k=1}^{\infty}\vecz(F_k),
\end{align*}
\end{small}
which implies that $\sum_{k=1}^{\infty}\vecz(F_k) = \mathbb{1}_{n^2\times 1}$.
\end{proof}

\subsection{Well-posedness of the optimization problem}
We here show that the function $\Hrt$ is continuous over the compact set
$\mathcal{P}_{\mathcal{G},\bm{\pi}}^\epsilon$. Then, by the extreme value
theorem, $\Hrt$ has a (possibly non-unique) maximum point in the set and
thus Problem \ref{prob:weightedentropy} is well-posed.


\begin{lemma}\label{lemma:continuity}
(Continuity of the return time entropy function) Given the compact set
  $\mathcal{P}_{\mathcal{G},\bm{\pi}}^\epsilon$, the following statements
  hold:
  \begin{enumerate}
  \item there exist constants $\lambda_{\max}\in(0,1)$ and $c>0$ such that
    \begin{equation}\nonumber
      F_k(i,i)\leq c\lambda_{\max}^k, \quad\text{for all }
      k\in\mathbb{Z}_{>0}, i\in\until{n};
    \end{equation}

  \item the return time entropy functions
    $\mathbb{H}(T_{ii})$, $i\in\until{n}$, and $\Hrt(P)$ are continuous on
    the compact set $\mathcal{P}_{\mathcal{G},\bm{\pi}}^\epsilon$; and 

  \item Problem~\ref{prob:weightedentropy} is well-posed in
    the sense that a global optimum exists.
\end{enumerate}
\end{lemma}

\begin{proof}
Regarding (i), for $k\geq w_{M}+1$, since the spectral radius $\rho(\Psi)$ is a continuous function of $\Psi$ \cite[Example 7.1.3]{CDM:01}, where $\Psi$ is given in \eqref{eq:sysmatrix}, and $\Psi$ is a continuous function of $P$, $\rho(\Psi)$ is a continuous function of $P$. Hence, by Lemma~\ref{lemma:systempropertiestravel}(ii) and the extreme value theorem, there exists a $\rho_{\max}<1$ such that
\begin{equation*}
\rho_{\max} = \max\limits_{P \in\mathcal{P}_{\mathcal{G},\bm{\pi}}^\epsilon} \rho(\Psi)<1.
\end{equation*}
Therefore, for $k\geq w_{M}+1$ and $i\in\until{n}$, by Lemma~\ref{lemma:SolutionBound}, there exist $c_1>0$ and $\rho_{\max}<\lambda_{\max}<1$ such that
\begin{align*}
F_k(i,i)&\leq\|\vecz{(\tilde{F}_{k-w_{\max}+1})}\|_\infty\\
&=\|(\Psi)^{k-w_{\max}}\vecz(\tilde{F}_{1})\|_\infty\\
&\leq\|(\Psi)^{k-w_{\max}}\|_\infty\|\vecz(\tilde{F}_{1})\|_\infty\\
&\leq c_1\lambda_{\max}^{k-w_{\max}}=\frac{c_1}{\lambda_{\max}^{w_{\max}}}\lambda_{\max}^{k}.
\end{align*}
Let $c=\max\{\frac{c_1}{\lambda_{\max}^{w_{\max}}},\frac{1}{\lambda_{\max}^{w_{\max}}}\}$, then we have for $k\geq w_{M}+1$,
\begin{align*}
F_k(i,i)\leq\frac{c_1}{\lambda_{\max}^{w_{\max}}}\lambda_{\max}^{k}<c\lambda_{\max}^{k}.
\end{align*}
For $k\leq w_{M}$,
\begin{align*}
c\lambda_{\max}^k\geq c\lambda_{\max}^{w_{\max}}\geq 1\geq F_{k}(i,i).
\end{align*}
Therefore, we have (i).

Regarding (ii), due to (i), there exists a positive integer $K$ that does
not depend on the elements of $\mathcal{P}_{\mathcal{G},\bm{\pi}}^\epsilon$
such that when $k\geq K$, $c\lambda_{\max}^k\leq e^{-1}$. Since $x\mapsto
-x\log x$ is an increasing function for $x\in[0,e^{-1}]$, when $k\geq K$,
\begin{equation*}
-F_k(i,i)\log F_k(i,i)\leq-c\lambda_{\max}^k\log (c\lambda_{\max}^k)\coloneqq M_k.
\end{equation*}
For $k< K$, $-F_k(i,i)\log F_k(i,i)\leq e^{-1}\coloneqq M_k$. Then
\begin{small}
\begin{equation*}
\sum_{k=1}^{K-1} M_k =\frac{K-1}{e},
\end{equation*}
\end{small}
and
\begin{small}
\begin{align}\label{eq:tailbound}
\sum_{k=K}^\infty M_k&=-\sum_{k=K}^\infty c\lambda_{\max}^k\log (c\lambda_{\max}^k)\nonumber\\
&=-c\log c\sum_{k=K}^\infty \lambda_{\max}^k-c\log (\lambda_{\max})\sum_{k=K}^\infty k\lambda_{\max}^k\nonumber\\
&=-c\Big(\frac{\lambda_{\max}^K}{1-\lambda_{\max}}\log (c\lambda_{\max}^K) \nonumber\\
&\qquad\quad+\frac{\lambda_{\max}^{K+1}}{(1-\lambda_{\max})^2}\log (\lambda_{\max})\Big).
\end{align}
\end{small}
Hence,
\begin{align*}
\sum_{k=1}^\infty M_k&=\sum_{k=1}^{K-1} M_k+\sum_{k=K}^\infty M_k<\infty,
\end{align*}
which holds for any $i$ and any transition matrix in the compact set $\mathcal{P}_{\mathcal{G},\bm{\pi}}^\epsilon$. By Lemma~\ref{lemma:Weierstrass}, the series $-\sum_{k=1}^\infty F_k(i,i)\log F_k(i,i)$ converges uniformly. Since the the limit of a
uniformly convergent series of continuous function is continuous
\cite[Theorem 7.12]{WR:76}, $\mathbb{H}(T_{ii})$ is a continuous function on $\mathcal{P}_{\mathcal{G},\bm{\pi}}^\epsilon$. Finally, $\Hrt(P)$ is a finite weighted sum of continuous functions $\mathbb{H}(T_{ii})$, thus $\Hrt(P)$ is a continuous function.

Regarding (iii), because $\Hrt$ is a continuous function over the compact set
$\mathcal{P}_{\mathcal{G},\bm{\pi}}^\epsilon$, the extreme value theorem
ensures that Problem \ref{prob:weightedentropy} admits a global optimum
solution (possibly non-unique) and is therefore well-posed.
\end{proof}

\subsection{Optimal solution for complete graphs with unitary travel times}
We here provide (1) an upper bound for the return time entropy with unitary
travel times based on the principle of maximum entropy and (2) the optimal
solution to Problem~\ref{prob:weightedentropy} for the complete graph case
with unitary travel times.

\begin{lemma}(Maximum achieved return time entropy in a complete graph with unitary weights)
Given a strongly connected graph $\mathcal{G}$ with unitary weights and the compact set $\mathcal{P}_{\mathcal{G},\bm{\pi}}^\epsilon$,
\begin{enumerate}
\item the return time entropy function is upper bounded by
\begin{align*}
\Hrt(P)\leq-\sum_{i=1}^n(\pi_i \log {\pi_i} + (1-{\pi_i})\log(1-{\pi_i}));
\end{align*}
\item when the graph $\mathcal{G}$ is complete, the upper
  bound is achieved and the transition matrix that maximizes the
  return time entropy $\Hrt(P)$ is given by
  $P=\vectorones[n]\bm{\pi}^\top$.
\end{enumerate}
\end{lemma}


\begin{proof}
Regarding (i), by Remark \ref{rmk:expectedreturntime}, in the case of unitary travel times, we have
$\E[T_{ii}]=1/\pi_i$. Thus, $T_{ii}$ is a discrete random variable
with fixed expectation, whose entropy is bounded as shown in
Lemma~\ref{lemma:MaximumEntropy}. For any transition matrix $P\in\mathcal{P}_{\mathcal{G},\bm{\pi}}^\epsilon$, the return time entropy function $\Hrt(P)$ satisfies
\begin{small}
\begin{align*}
\Hrt(P)&=\sum_{i=1}^n\pi_i\mathbb{H}(T_{ii})\leq\sum_{i=1}^n\pi_i\max_{T_{ii}}\{\mathbb{H}(T_{ii})\} \\
&=\sum_{i=1}^n\pi_i\big(\frac{1}{\pi_i} \log \frac{1}{\pi_i} -(\frac{1}{\pi_i}-1)\log(\frac{1}{\pi_i}-1)\big)\\
&=-\sum_{i=1}^n\big(\pi_i \log {\pi_i} + (1-{\pi_i})\log(1-{\pi_i})\big),
\end{align*}
\end{small}
where the third line uses \eqref{eq:maximumentropy}.

Regarding (ii), when the graph is complete and $P=\vectorones[n]\bm{\pi}^\top$, the return time $T_{ii}$ follows the geometric distribution:
\begin{equation}\nonumber
\prob(T_{ii}=k)=\pi_i(1-\pi_i)^{k-1}.
\end{equation}
Then by Lemma~\ref{lemma:MaximumEntropy}, we obtain the results.
\end{proof}

\subsection{Relations with the entropy rate of Markov chains}
Given an irreducible Markov chain $P$ with $n$ nodes and stationary
distribution $\bm{\pi}$, the \emph{entropy rate of $P$} is given by
\begin{equation*}
  \Hrate(P)=-\sum_{i=1}^n\pi_i\sum_{j=1}^np_{ij}\log p_{ij}.
\end{equation*}
We next study the relationship between the return time entropy $\Hrt$ with
unitary travel times and the entropy rate $\Hrate$.

\begin{theorem}\label{thm:relations}
(Relations between the return time entropy with unitary travel times and the
  entropy rate) For all $P$ in the compact set
  $\mathcal{P}_{\mathcal{G},\bm{\pi}}^\epsilon$ where $\mathcal{G}$ has unitary travel times, the return time entropy
  $\Hrt(P)$ and the entropy rate $\Hrate(P)$ satisfy
  \begin{equation}\label{eq:relations}
    \Hrate(P)\leq \Hrt(P) \leq n\Hrate(P),
  \end{equation}
  where $n$ is the number of nodes in the graph $\mathcal{G}$.
\end{theorem}
We prove this theorem The proof of the following theorem follows from Lemmas~
\ref{lemma:upperbounds_entropy} and Lemma
\ref{lemma:lowerbound_entropyrate} below.

\begin{remark}
Theorem \ref{thm:relations} establishes a large gap, possibly of size $O(n)$, between $\Hrate(P)$ and $\Hrt(P)$ and,
  thereby, optimizing $\Hrate$ and $\Hrt$ are two
  different matters altogether.
\end{remark}

First, we show that the return time entropy is upper bounded by $n$ times
of the entropy rate. As in~\cite{LE-TMC:93}, we define a \emph{Markov
  trajectory from state $i$ to state $j$} to be a path with initial state
$i$, final state $j$, and no intervening state equal to $j$. Let
$\mathcal{T}_{ij}$ be the set of all Markov trajectories from state $i$ to
state $j$.  Let $\mathbb{P}(\ell)$ denote the probability of a Markov
trajectory $\ell\in\mathcal{T}_{ij}$; clearly
$\sum_{\ell\in\mathcal{T}_{ij}}\mathbb{P}(\ell)=1$. Let $L_{ij}$ be the
Markov trajectory random variable that takes value $\ell$ in
$\mathcal{T}_{ij}$ with probability $\mathbb{P}(\ell)$. Finally, we define
the entropy of $L_{ij}$ by
\begin{equation*}
  \mathbb{H}(L_{ij})=-\sum_{\ell\in\mathcal{T}_{ij}}\mathbb{P}(L_{ij}=\ell)\log\mathbb{P}(L_{ij}=\ell).
\end{equation*}
\begin{lemma}(Entropy of Markov trajectories \cite[Theorem 1]{LE-TMC:93})\label{lemma:MarkovTrajectories}
For an irreducible Markov chain with transition matrix $P$, the entropy
$\mathbb{H}(L_{ii})$ of the random Markov trajectory from state $i$ back to
state $i$ is given by
\begin{equation*}
  \mathbb{H}(L_{ii})=\frac{\Hrate(P)}{{\pi}_i}.
\end{equation*}
\end{lemma}

Through the entropy of the Markov trajectories, we are able to establish the upper bound of the return time entropy in~\eqref{eq:relations}.
\begin{lemma}(Upper bound of the return time entropy by $n$ times of the entropy rate)\label{lemma:upperbounds_entropy}
Given the compact set $\mathcal{P}_{\mathcal{G},\bm{\pi}}^\epsilon$,
\begin{enumerate}
\item the return time entropy is upper bounded by
\begin{equation}\label{eq:JHconnection}
\Hrt(P)\leq n\Hrate(P),  \quad\text{for all } P\in\mathcal{P}_{\mathcal{G},\bm{\pi}}^\epsilon;
\end{equation}
\item the equality in~\eqref{eq:JHconnection} holds if and only if any
  node of the graph $\mathcal{G}$ has the property that all distinct first return paths have
  different length, i.e., the return paths are distinguishable by their lengths, and in this case,
\begin{equation*}
\argmax_{P\in\mathcal{P}_{\mathcal{G},\bm{\pi}}^\epsilon}\Hrt(P)=\argmax_{P\in\mathcal{P}_{\mathcal{G},\bm{\pi}}^\epsilon}\Hrate(P).
\end{equation*}
\end{enumerate}
\end{lemma}

\begin{proof}
Regarding (i), the return time random variable $T_{ii}$ is defined by lumping the trajectories in $\mathcal{T}_{ii}$ with the same length,
\begin{equation}\label{eq:time_trajectory}
\mathbb{P}(T_{ii}=k)=\sum_{\ell\in\mathcal{T}_{ii},|\ell|=k}\mathbb{P}(L_{ii}=\ell),
\end{equation}
where $|\ell|$ denotes the length of the path $\ell$. Note that
\begin{small}
\begin{align}\label{eq:inequality_upperbound}
-\mathbb{P}(&T_{ii}=k)\log\mathbb{P}(T_{ii}=k)\nonumber\\
&=-\big(\sum_{\ell\in\mathcal{T}_{ii},|\ell|=k}\mathbb{P}(L_{ii}=\ell)\big)\log\big(\sum_{\ell\in\mathcal{T}_{ii},|\ell|=k}\mathbb{P}(L_{ii}=\ell)\big)\nonumber\\
&\leq-\sum_{\ell\in\mathcal{T}_{ii},|\ell|=k}\mathbb{P}(L_{ii}=\ell)\log\mathbb{P}(L_{ii}=\ell),
\end{align}
\end{small}where we used that $(x+y)\log(x+y)\geq x\log x+y\log y$ for $x,y\geq0$. Since both the return time entropy and the entropy of Markov trajectories are absolutely convergent, we have
\begin{small}
\begin{align*}
\mathbb{H}(T_{ii})&=-\sum_{k=1}^\infty\mathbb{P}(T_{ii}=k)\log\mathbb{P}(T_{ii}=k)\\
&\leq-\sum_{k=1}^\infty\sum_{\ell\in\mathcal{T}_{ii},|\ell|=k}\big(\mathbb{P}(L_{ii}=\ell)\log\mathbb{P}(L_{ii}=\ell)\big)\\
&=\mathbb{H}(L_{ii}),
\end{align*}
\end{small}
which along with Lemma~\ref{lemma:MarkovTrajectories} imply
\begin{equation*}
\Hrt(P)\leq n\Hrate(P).
\end{equation*}

Regarding (ii), the inequality in~\eqref{eq:JHconnection} comes from the inequality \eqref{eq:inequality_upperbound}. If any node of the graph $\mathcal{G}$ has the property that all distinct first return paths have different length, then the summation on the right hand side of \eqref{eq:time_trajectory} only has one term and the inequality in \eqref{eq:inequality_upperbound} becomes an equality. On the other hand, if for some node of $\mathcal{G}$, there are distinct return paths that have the same length, then one needs to lump the paths with the same length and the inequality in \eqref{eq:inequality_upperbound} becomes strict. Moreover, if the equality holds, then $\Hrt(P)$ is a constant $n$ times of $\Hrate(P)$ and thus they have the same maximizer.
\end{proof}

\begin{example}
  For the two-node case in Examples \ref{example:closedform}(i), the return
  time entropy is twice the entropy rate. This is not a coincidence since
  the 2-node complete graph satisfies the property in Lemma
  \ref{lemma:upperbounds_entropy}(ii). Figure~\ref{fig:examplegraph}
  illustrates a graph with $4$ nodes that also satisfies the property in
  Lemma~\ref{lemma:upperbounds_entropy}(ii).
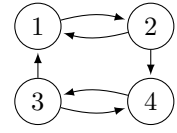
\begin{figure}
\centering
\begin{tikzpicture}
        \node[state,minimum size = 0.2cm] at (0, 1) (nodeone) {\text{$1$}};
        \node[state,minimum size = 0.2cm] at (1.5, 1)     (nodetwo)     {\text{$2$}};
        \node[state,minimum size = 0.2cm] at (0, 0)     (nodethree)     {\text{$3$}};
        \node[state,minimum size = 0.2cm] at (1.5, 0)      (nodefour)    {\text{$4$}};
        \draw[every loop,
              auto=right,
              >=latex,
              ]
            (nodefour)     edge[bend right=15, auto=right] node {} (nodethree)
            (nodetwo)     edge[bend left=15, auto=right] node {} (nodeone)
            (nodetwo)     edge[bend right=0, auto=right] node {} (nodefour)

            (nodethree)     edge[bend right=15]            node {} (nodefour)
            (nodethree)     edge[bend right=0]            node {} (nodeone)

            (nodeone) edge[bend left=15, auto=right] node {} (nodetwo);
\end{tikzpicture}
\caption{An example graph that satisfies the property in Lemma~\ref{lemma:upperbounds_entropy}(ii)}\label{fig:examplegraph}
\end{figure}

\end{example}

In the rest of this subsection, we show that the return time entropy is lower bounded by the entropy rate as shown in~\eqref{eq:relations}.
\begin{lemma}(Lower bound of the return time entropy by the entropy rate)\label{lemma:lowerbound_entropyrate}
Given the compact set $\mathcal{P}_{\mathcal{G},\bm{\pi}}^\epsilon$,
\begin{enumerate}
\item the return time entropy is lower bounded by
\begin{equation}\label{eq:JHlowerbound}
\Hrt(P)\geq \Hrate(P),  \quad\text{for all }  P\in\mathcal{P}_{\mathcal{G},\bm{\pi}}^\epsilon;
\end{equation}
\item the equality in~\eqref{eq:JHlowerbound} holds if and only if $P$ is a permutation matrix.
\end{enumerate}
\end{lemma}
\begin{proof}
Regarding (i), note that the first hitting time $T_{ij}$ from state $i$ to state $j$ as defined in~\eqref{eq:passagetime} is a random variable , whose entropy is $\mathbb{H}(T_{ij})$. Then by definition, we have in the case of unitary travel times,
\begin{small}
\begin{align*}
\mathbb{H}&(T_{ij})=-\sum\limits_{k=1}^\infty \mathbb{P}(T_{ij}=k)\log \mathbb{P}(T_{ij}=k)\\
&=-p_{ij}\log p_{ij} - (\sum_{k_1\neq j}p_{ik_1}p_{k_1j})\log(\sum_{k_1\neq j}p_{ik_1}p_{k_1j}) \\
&\quad-(\sum_{k_1,k_2\neq j}p_{ik_1}p_{k_1k_2}p_{k_2j})\log(\sum_{k_1,k_2\neq j}p_{ik_1}p_{k_1k_2}p_{k_2j})\\
&\quad -\cdots\\
&\quad-(\sum_{k_1\cdots k_m\neq j}p_{ik_1}\cdots p_{k_mj})\log(\sum_{k_1\cdots k_m\neq j}p_{ik_1}\cdots p_{k_mj})\\
&\quad-\cdots.
\end{align*}
\end{small}
Since $x\mapsto -x\log x$ is a concave function, for $x_i\geq0$ and for
coefficients $\alpha_i\geq 0$ satisfying $\sum_{i=1}^n\alpha_i=1$, we have
\begin{equation}\label{eq:concavefunction}
  -(\sum_{i=1}^n\alpha_ix_i)\log(\sum_{i=1}^n\alpha_ix_i)\geq -\sum_{i=1}^n \alpha_i (x_i\log x_i).
\end{equation}
Thus, for $m\geq 1$,
\begin{small}
\begin{align}\label{eq:entropyconcavity}
-\mathbb{P}(&T_{ij}=m+1)\log \mathbb{P}(T_{ij}=m+1)\nonumber\\
&=-(\sum_{k_1\cdots k_m\neq j}p_{ik_1}\cdots p_{k_mj})\log(\sum_{k_1\cdots k_m\neq j}p_{ik_1}\cdots p_{k_mj})\nonumber\\
&=- (\sum_{k_1\neq j}p_{ik_1}\sum_{k_2\cdots k_m\neq j}p_{k_1k_2}\cdots p_{k_mj}+p_{ij}\cdot0)\nonumber\\
&\qquad \cdot\log (\sum_{k_1\neq j}p_{ik_1}\sum_{k_2\cdots k_m\neq j}p_{k_1k_2}\cdots p_{k_mj} + p_{ij}\cdot0)\nonumber\\
&\geq- \sum_{k_1\neq j}p_{ik_1}(\sum_{k_2\cdots k_m\neq j}p_{k_1k_2}\cdots p_{k_mj}\nonumber\\
&\qquad\qquad\qquad\quad \cdot\log (\sum_{k_2\cdots k_m\neq j}p_{k_1k_2}\cdots p_{k_mj}))\nonumber\\
&=-\sum_{k_1\neq j}p_{ik_1}\mathbb{P}(T_{k_1j}=m)\log \mathbb{P}(T_{k_1j}=m),
\end{align}
\end{small}where the inequality uses equation~\eqref{eq:concavefunction}. Summing
both sides of~\eqref{eq:entropyconcavity} over $m$ for $m\geq 1$, we
have
\begin{multline}\label{eq:passagetimeentropy}
\mathbb{H}(T_{ij}) \geq -p_{ij}\log p_{ij} + \sum\nolimits_{k_1\neq j}p_{ik_1}\mathbb{H}(T_{k_1j}) \\
=-p_{ij}\log p_{ij} + \sum\nolimits_{k_1=1}^np_{ik_1}\mathbb{H}(T_{k_1j})-p_{ij}\mathbb{H}(T_{jj}).
\end{multline}
Let $\mathbb{H}(T)$ be a matrix whose $(i,j)$-th element is
$\mathbb{H}(T_{ij})$. Then equation~\eqref{eq:passagetimeentropy} can
be put in the matrix form
\begin{align}\label{eq:passagetimeentropy_matrix}
\mathbb{H}(T) \geq -P\circ\log P + P\mathbb{H}(T)-P[\mathbb{H}(T)],
\end{align}
where the inequality and the $\log$ function are entry-wise. Multiplying $\bm{\pi}^\top$ from the left and $\mathbb{1}_n$ from the right on both sides of~\eqref{eq:passagetimeentropy_matrix}, we have
\begin{align*}
\bm{\pi}^\top [\mathbb{H}(T)]\mathbb{1}_n\geq -\bm{\pi}^\top (P\circ\log P) \mathbb{1}_n,
\end{align*}
which is $\Hrt(P)\geq \Hrate(P)$.

Regarding (ii), if $P$ is a permutation matrix, then $\Hrt(P)=\Hrate(P)=0$. On the other hand, if $P$ is not a permutation matrix, then there exist $2$ or more nonzero elements on at least one row of $P$. In this case, the inequality in~\eqref{eq:entropyconcavity} is strict for that row for some $m$, which carries over to \eqref{eq:passagetimeentropy}. Thus, $\Hrt(P)>\Hrate(P)$.

\end{proof}


\section{Truncated return time entropy and its optimization via gradient descent}\label{sec:approximation}
We now introduce the truncated and conditional return time entropy and
setup a gradient descent algorithm.

\subsection{The truncated and conditional return time entropies}
In practical applications, we may discard events occurring with extremely
low probability. In what follows, we study the return time distribution and
its entropy conditioned upon the event that the return time is upper
bounded.  We first introduce a \emph{truncation accuracy} parameter
$0<\eta\ll1$ that upper bounds the cumulative probabilities of very large
return times and we define a \emph{duration} $N_\eta\in\mathbb{Z}_{>0}$ by
\begin{equation}\label{eq:duration}
  N_\eta = \ceil[\Big]{\frac{w_{\max}}{\eta\pi_{\min}}} -1,
\end{equation}
where
$\pi_{\min}=\min_{i\in\until{n}}\{\pi_i\}$ and $\ceil{\cdot}$ is the ceiling function. It is an immediate
consequence of the Markov's inequality that, given the fixed
stationary distribution $\bm{\pi}$, for all $i\in\until{n}$,
\begin{equation*}
  \mathbb{P}(T_{ii}\geq N_\eta+1)\leq \frac{\mathbb{E}[T_{ii}]}{N_\eta+1}\leq\frac{w_{\max}}{\pi_i(N_\eta+1)}\leq \eta,
\end{equation*}
where we used \eqref{eq:ETii}
\begin{equation*}
  \mathbb{E}[T_{ii}]=\frac{\bm{\pi}^\top(P\circ W)\mathbb{1}_n}{\pi_i}\leq\frac{w_{\max}}{\pi_i}.
\end{equation*}

We now define the conditional return time and its entropy.

\begin{definition}(Conditional return time and its entropy)
  Given $P\in \mathcal{P}_{\mathcal{G},\bm{\pi}}^\epsilon$ and a duration
  $N_\eta$, the conditional return time $T_{ii}\,|\,T_{ii}\leq N_\eta$ of
  state $i$ is defined by
  \begin{small}
  \begin{align*}
    T_{ii}\,|\,T_{ii}\leq N_\eta&=\min\Big\{\sum_{k'=0}^{k-1}w_{X_{k'}X_{k'+1}}\,|\,\sum_{k'=0}^{k-1}w_{X_{k'}X_{k'+1}}\leq N_\eta,\\
    &\quad\quad\quad\quad\quad\quad\quad\quad\quad\quad X_0=i,X_k=i, k\geq1\Big\}.
  \end{align*}
  \end{small}
with probability mass function
\begin{equation}\nonumber
\mathbb{P}(T_{ii}=k\,|\,T_{ii}\leq N_\eta)=\frac{F_k(i,i)}{\sum_{k=1}^{N_\eta} F_k(i,i)}.
\end{equation}
Moreover, the conditional return time entropy function
$\subscr{\Hrt}{cond,$\eta$}:\mathcal{P}_{\mathcal{G},\bm{\pi}}^\epsilon
\mapsto \real_{\geq 0}$ is
\begin{small}
\begin{align*}
\subscr{\Hrt}{cond,$\eta$}(P)&=\sum_{i=1}^n\pi_i\mathbb{H}(T_{ii}\,|\,T_{ii}\leq N_\eta)\\
&=-\sum_{i=1}^n\pi_i\sum\limits_{k=1}^{N_\eta} \frac{F_k(i,i)}{\sum\limits_{k=1}^{N_\eta} F_k(i,i)}\log \frac{F_k(i,i)}{\sum\limits_{k=1}^{N_\eta} F_k(i,i)}.
\end{align*}
\end{small}
\end{definition}

Given the duration $N_\eta$, $\subscr{\Hrt}{cond,$\eta$}(P)$ is
a finite sum of continuously differentiable functions and thus more
tractable than the original return time entropy function
$\Hrt(P)$. Next, we introduce a truncated entropy that is even
simpler to evaluate.

\begin{definition}(Truncated return time entropy function)\label{def:Jtrunc}
Given a compact set $\mathcal{P}_{\mathcal{G},\bm{\pi}}^\epsilon$ and
the duration $N_\eta$, define the \emph{truncated return time entropy
  function}
$\subscr{\Hrt}{trunc,$\eta$}:\mathcal{P}_{\mathcal{G},\bm{\pi}}^\epsilon
\mapsto \real_{\geq 0}$ by
\begin{equation*}
\subscr{\Hrt}{trunc,$\eta$}(P)=-\sum_{i=1}^n\pi_i\sum\limits_{k=1}^{N_\eta} F_k(i,i)\log F_k(i,i).
\end{equation*}
\end{definition}

The following lemma shows that, for small $\eta$, the truncated return
time entropy $\subscr{\Hrt}{trunc,$\eta$}(P)$ is a good approximation for the
conditional return time entropy $\subscr{\Hrt}{cond,$\eta$}(P)$. Furthermore, when $\eta$ is sufficiently small, the truncated return time entropy $\subscr{\Hrt}{trunc,$\eta$}(P)$ is also a good approximation for the original return time entropy function $\Hrt(P)$.

\begin{lemma}\label{lemma:trunccondrelation}
  (Approximation bounds)
  Given $P\in \mathcal{P}_{\mathcal{G},\bm{\pi}}^\epsilon$ and the truncation accuracy $\eta$, we have
\begin{enumerate}
\item the conditional return time entropy is related to the truncated return time entropy by
\begin{equation}\label{eq:trunccondrelation}
    \subscr{\Hrt}{trunc,$\eta$}(P)+ \log (1-\eta)<\subscr{\Hrt}{cond,$\eta$}(P)<\frac{\subscr{\Hrt}{trunc,$\eta$}(P)}{1-\eta};
\end{equation}

\item $\Hrt(P)\geq \subscr{\Hrt}{trunc,$\eta$}(P)$ holds trivially and if
  \begin{equation}\label{eq:etacondition}
  \eta \leq  \frac{w_{\max}\log\lambda_{\max}}{\pi_{\min}(\log\lambda_{\max}-\log c-1)},
  \end{equation}
  then
  \begin{equation}\label{eq:errorbound}
    \Hrt(P)-\subscr{\Hrt}{trunc,$\eta$}(P) \leq \frac{c \log(\lambda_{\max}^{-1}) }{(1-\lambda_{\max})^2} ( 1+ N_\eta ) \lambda_{\max}^{N_\eta} ,
\end{equation}
where $c$ and $\lambda_{\max}$ are given as in Lemma~\ref{lemma:continuity}(i);

\item $\displaystyle
  \Hrt(P)
    =
    \lim_{\eta\to 0^+} \subscr{\Hrt}{cond,$\eta$}(P)
    =
    \lim_{\eta\to 0^+}  \subscr{\Hrt}{trunc,$\eta$}(P)    $.

  \end{enumerate}
\end{lemma}

\begin{proof}
Regarding (i), for $\subscr{\Hrt}{cond,$\eta$}(P)$, we have
\begin{small}
\begin{align*}
\subscr{\Hrt}{cond,$\eta$}(P)&=-\sum_{i=1}^n\pi_i\sum\limits_{k=1}^{N_\eta} \frac{F_k(i,i)}{\sum\limits_{k=1}^{N_\eta} F_k(i,i)}\log \frac{F_k(i,i)}{\sum\limits_{k=1}^{N_\eta} F_k(i,i)}\\
&=-\sum_{i=1}^n\pi_i\big(\frac{\sum\limits_{k=1}^{N_\eta} F_k(i,i)\log F_k(i,i)}{\sum\limits_{k=1}^{N_\eta} F_k(i,i)}- \log {\sum\limits_{k=1}^{N_\eta} F_k(i,i)}\big).
\end{align*}
\end{small}
On one hand,
\begin{small}
\begin{align}\label{eq:Jcondupper}
\subscr{\Hrt}{cond,$\eta$}(P)&>-\sum_{i=1}^n\pi_i\big(\sum\limits_{k=1}^{N_\eta} F_k(i,i)\log F_k(i,i)- \log {\sum\limits_{k=1}^{N_\eta} F_k(i,i)}\big)\nonumber\\
&\geq-\sum_{i=1}^n\pi_i\sum\limits_{k=1}^{N_\eta} F_k(i,i)\log F_k(i,i)+ \log (1-\eta).
\end{align}
\end{small}
On the other hand,
\begin{small}
\begin{align}\label{eq:Jcondlower}
\begin{split}
\subscr{\Hrt}{cond,$\eta$}(P)&<-\sum_{i=1}^n\pi_i\frac{1}{\sum\limits_{k=1}^{N_\eta} F_k(i,i)}\sum\limits_{k=1}^{N_\eta} F_k(i,i)\log F_k(i,i)\\
&\leq-\frac{1}{1-\eta}\sum_{i=1}^n\pi_i\sum\limits_{k=1}^{N_\eta} F_k(i,i)\log F_k(i,i).
\end{split}
\end{align}
\end{small}
Combining \eqref{eq:Jcondupper} and \eqref{eq:Jcondlower}, we have \eqref{eq:trunccondrelation}.

Regarding (ii), if $\eta$ satisfies \eqref{eq:etacondition}, we have
$c\lambda_{\max}^{N_\eta}\leq e^{-1}$. Then, following the same
arguments as in the proof of Lemma~\ref{lemma:continuity}(ii) and
replacing $K$ in \eqref{eq:tailbound} with $N_\eta$, we have
\begin{small}
\begin{align*}
    &\Hrt(P)-\subscr{\Hrt}{trunc,$\eta$}(P) \nonumber\\
    &\leq -c\left(\frac{\lambda_{\max}^{N_\eta}}{1-\lambda_{\max}}\log (c\lambda_{\max}^{N_\eta})+\frac{\lambda_{\max}^{N_\eta+1}}{(1-\lambda_{\max})^2}\log (\lambda_{\max})\right)\\
 &\leq - \frac{c \lambda_{\max}^{N_\eta} }{(1-\lambda_{\max})^2} \left( N_\eta \log(\lambda_{\max})+ \lambda_{\max}\log (\lambda_{\max})+ \log (c) \right) \\
 &\leq - \frac{c \lambda_{\max}^{N_\eta} }{(1-\lambda_{\max})^2} \left( N_\eta \log(\lambda_{\max})+ \log (\lambda_{\max}) \right) \\
&= \frac{c \log(\lambda_{\max}^{-1}) }{(1-\lambda_{\max})^2} ( 1+ N_\eta ) \lambda_{\max}^{N_\eta} .   \label{eq:errorbound}
\end{align*}
\end{small}

Regarding (iii), the results follow from \eqref{eq:trunccondrelation} and \eqref{eq:errorbound}, respectively. Specifically, in \eqref{eq:errorbound}, since $0<\lambda_{\max}<1$, the error
$\Hrt(P) - \subscr{\Hrt}{trunc,$\eta$}(P)$ goes to $0$
exponentially fast as $\eta$ goes to $0$ ($N_\eta \to \infty$).
\end{proof}

\subsection{The gradient of the truncated return time entropy}
Lemma~\ref{lemma:trunccondrelation} establishes how
$\subscr{\Hrt}{trunc,$\eta$}(P)$ is a good approximation to both of
$\Hrt(P)$ and $\subscr{\Hrt}{cond,$\eta$}(P)$. Since it is also easier to
compute $\subscr{\Hrt}{trunc,$\eta$}(P)$ than the other two quantities, we
focus on optimizing $\subscr{\Hrt}{trunc,$\eta$}(P)$ by computing its
gradient.

For $k\in\mathbb{Z}_{>0}$, define $G_k=\frac{\partial \vecz(F_k)}{\partial
  \vecz(P)}\in\mathbb{R}^{n^2\times n^2}$ and note
\begin{equation}\label{eq:Gkanotherform}
{G_k} = \begin{bmatrix}
\frac{\partial \vecz(F_k)}{\partial p_{11}}&
\frac{\partial \vecz(F_k)}{\partial p_{21}}&
 \cdots &
\frac{\partial \vecz(F_k)}{\partial p_{(n-1)n}}&
\frac{\partial \vecz(F_k)}{\partial p_{nn}}
\end{bmatrix}.
\end{equation}

\begin{lemma}(Gradient of the truncated return time entropy function)\label{lemma:gradient}
  Given $P\in\mathcal{P}_{\mathcal{G},\bm{\pi}}^\epsilon$, the matrix
  sequence $G_k$ in \eqref{eq:Gkanotherform} satisfies the iteration for
  $k\in\mathbb{Z}_{>0}$,

  \begin{align}\label{eq:GradientRecursive}
    G_{k} &= [\vecz(\mathbf{1}_{\{k\mathbb{1}_n\mathbb{1}_n^\top=W\}})] +\sum_{i=1}^{w_{\max}}\Phi_i G_{k-i}\nonumber\\
    &\quad+ \sum_{i=1}^n\sum_{j=1}^n (E_{j}F_{k-w_{ij}}^\top\otimes I_n)[\vecz(\mathbb{e}_{i}\mathbb{e}_{j}^\top)]\mathbf{1}_{\{w_{ij}>0\}},
  \end{align}
   where the initial conditions are $G_{k}=\mathbb{0}_{n^2\times n^2}$ for $k\leq0$. Moreover, the vectorization of the gradient of
  $\subscr{\Hrt}{trunc,$\eta$}$ satisfies
\begin{small}
\begin{multline}\label{eq:Gradient}
  \vecz\Big(\frac{\partial \subscr{\Hrt}{trunc,$\eta$}(P)}{\partial P}\Big) = \\
  -\sum_{i=1}^n\pi_i\sum_{k=1}^{N_\eta}
  \frac{\partial \big(F_k(i,i)\log F_k(i,i)\big)}{\partial F_k(i,i)}
  G_k^\top\mathbb{e}_{(i-1)n+i},
\end{multline}
\end{small}
where $\mathbb{e}_{(i-1)n+i}\in \real^{n^2}$ and
\begin{small}
\begin{equation*}
    \frac{\partial F_k(i,i)\log F_k(i,i)}{\partial F_k(i,i)} =    \begin{cases}
      1+\log(F_k(i,i)), &\quad \text{if}\enspace F_k(i,i)>0, \\
          0, &\quad \text{if}\enspace F_k(i,i)= 0.
          \end{cases}
\end{equation*}
\end{small}
\end{lemma}

\begin{proof}
For $k\in\mathbb{Z}_{>0}$, according to \eqref{eq:dynamicswithtravel}, we have for $p_{uv}>0$,
\begin{small}
  \begin{align*}
    \frac{\partial \vecz(F_k)}{\partial p_{uv}}&=\vecz(\mathbb{e}_u\mathbb{e}_v^\top)\mathbf{1}_{\{k=w_{uv}\}}\\
    &\quad+(E_v\otimes\mathbb{e}_u\mathbb{e}_v^\top)\vecz(F_{k-w_{uv}})\\
    &\quad+\sum_{i=1}^n\sum_{j=1}^np_{ij}(E_j\otimes\mathbb{e}_i\mathbb{e}_j^\top)\frac{\partial \vecz(F_{k-w_{ij}})}{\partial p_{uv}},
\end{align*}
\end{small}
where the second term on the right hand side satisfies
\begin{align*}\nonumber
  (E_v\otimes\mathbb{e}_u\mathbb{e}_v^\top)\vecz(F_{k-w_{uv}})
  &=\vecz(\mathbb{e}_u\mathbb{e}_v^\top F_{k-w_{uv}}E_v)\\
  &=(E_vF_{k-w_{uv}}^\top\otimes I_n)\vecz(\mathbb{e}_{u}\mathbb{e}_v^\top).
\end{align*}
Stacking $\frac{\partial\vecz(F_{k})}{\partial p_{uv}}$'s in a matrix as \eqref{eq:Gkanotherform}, we obtain \eqref{eq:GradientRecursive}.

Since $\subscr{\Hrt}{trunc,$\eta$}(P)$ only involves $F_k(i,i)$ for $i=\until{n}$, we only need the corresponding columns in $G_k^\top$ to compute the gradient, which is realized by multiplying the standard unit vector as in \eqref{eq:Gradient}.
\end{proof}
\begin{remark}
  Iteration~\eqref{eq:GradientRecursive} is an exponentially stable
  discrete-time delayed linear system subject to and a finite number of impulse inputs and an exponentially vanishing input. Hence,
  the state $G_k\to0$ exponentially fast as $k\to\infty$.
\end{remark}

\subsection{Optimizing the truncated entropy via gradient projection}
Motivated by the previous analysis, we consider the following problem.
\begin{problem}\label{probLP}
(Maximization of the truncated return time entropy) Given a strongly
  connected directed graph $\mathcal{G}$ and the stationary distribution
  $\bm{\pi}$, pick a minimum edge weight $\epsilon>0$ and a truncation
  accurate parameter $\eta>0$, the maximization of the truncated return
  time entropy function is as follows.
\begin{align*}
& \text{maximize}
& &\subscr{\Hrt}{trunc,$\eta$}(P)\\
& \text{subject to}
&& P\in\mathcal{P}_{\mathcal{G},\bm{\pi}}^\epsilon
\end{align*}
\end{problem}

To solve numerically this nonlinear program, we exploit the results in
Lemma~\ref{lemma:gradient} and adopt the gradient projection method as
presented in~\cite[Chapter 2.3]{DPB:16}:
\smallskip
 \begin{algorithmic}[1]
   \STATE select: minimum edge weight $\epsilon\ll1$, truncation accuracy
   $\eta\ll 1$, and initial condition $P_0$ in
   $\mathcal{P}_{\mathcal{G},\bm{\pi}}^\epsilon$
   \FOR{iteration parameter $s=0:\text{(number-of-steps)}$}
   \STATE $\{G_k\}_{k\in\until{N_\eta}} :=$ solution to iteration~\eqref{eq:GradientRecursive} at $P_s$
   \STATE $\Delta_s := $ gradient of $\subscr{\Hrt}{trunc,$\eta$}(P_s)$ via equation~\eqref{eq:Gradient}
   \STATE $P_{s+1}:= \operatorname{projection}_{\mathcal{P}_{\mathcal{G},\bm{\pi}}^\epsilon}(P_s + \text{(step size)} \cdot  \Delta_s)$
   \ENDFOR
 \end{algorithmic}
 \smallskip

 We analyze the computational complexity of this algorithm. To compute
 step~\algostep{3}, we need to evaluate the right-hand side of
 equation~\eqref{eq:GradientRecursive} by computing three terms. For the
 first term, we need to do $m$ comparisons, where $m$ is the number of
 edges in the graph (i.e., the number of variables in the transition
 matrix), and it takes $O(m)$ elementary operations. For the second term,
 note that the matrices $\Phi_i\in\real^{n^2\times{n^2}}$ introduced in
 equation~\eqref{eq:phimatrix} can be precomputed and is block diagonal
 with $n$ blocks of size $n\times{n}$. Also note that
 $G_k\in\real^{n^2\times{n^2}}$ has only $m$ nonzero columns. Thus, we need
 $O(w_{\max}mn^3)$ operations.  For the third term, $F_k$ is updated by
 equation~\eqref{eq:augmentedsystem}, which requires $O(w_{\max}n^3)$ and
 is the main computational cost. Therefore, it takes $O(w_{\max}mn^3)$ to
 compute one update of iteration~\eqref{eq:GradientRecursive}.  Thus, it
 takes $O(N_{\eta}w_{\max}mn^3)$ elementary operations to complete
 step~\algostep{3}. In step~\algostep{5}, we need to solve a least square
 problem with linear equalities and inequalities constraints; which
 requires $O(m^3)$ \cite{SB-LV:04}.


\section{Numerical results}\label{sec:Simulation}
In this section, we provide numerical results on the computation of the
maximum return time entropy chain (Subsection~\ref{subsec:comp-MRE}) and
its application to robotic surveillance problems
(Subsection~\ref{subsec:app-robotics}).
We compute and compare three chains:
\begin{enumerate}
\item the Markov chain that maximizes the return time entropy (solution of
  Problem~\ref{prob:weightedentropy}), abbreviated as the
  \emph{MaxReturnEntropy chain}.  This chain may be computed for a directed
  graph with arbitrary integer-valued travel times.  Since we do not have a
  way to solve Problem~\ref{prob:weightedentropy} directly, the
  MaxReturnEntropy chain is approximated by the solution of
  Problem~\ref{probLP}, which is solved via the gradient projection
  method. Unless otherwise stated, we choose truncation accuracy $\eta =
  0.1$. Note that \eqref{eq:duration} is quite conservative and the actual
  probabilities being discarded is much less than $0.1$.

\item the Markov chain that maximizes the entropy rate, abbreviated as the
  \emph{MaxEntropyRate chain}. This chain can be computed for a directed
  graph with unitary weights via solving a convex program. Further, if the
  graph is undirected, the MaxEntropyRate chain can be computed efficiently
  using the method in \cite{MG-SJ-FB:17b};

\item the Markov chain that minimizes the (weighted) Kemeny constant,
  abbreviated as the \emph{MinKemeny chain}. This chain may be computed for
  a directed graph with arbitrary travel times via solving a nonlinear
  nonconvex program. We compute this chain using the solver implemented in
  the KNITRO/TOMLAB package.

\end{enumerate}

\subsection{Computation, comparison and intuitions}\label{subsec:comp-MRE}
We divide this subsection into two parts. In the first part, we first compare $3$ chains on graphs that have unitary travel times. We then summarize several observations in computing the MaxReturnEntropy chain. Finally, we visualize and plot the chains as well as the return time distributions. In the second part, we compare the MaxReturnEntropy chain with the MinKemeny chain on a realistic map taken from \cite[Section 6.2]{SA-EF-SLS:14} with travel times.

\begin{figure*}
x\begin{center}
\subfigure[MaxReturnEntropy chain on ring graph]{
\includegraphics[scale=0.38]{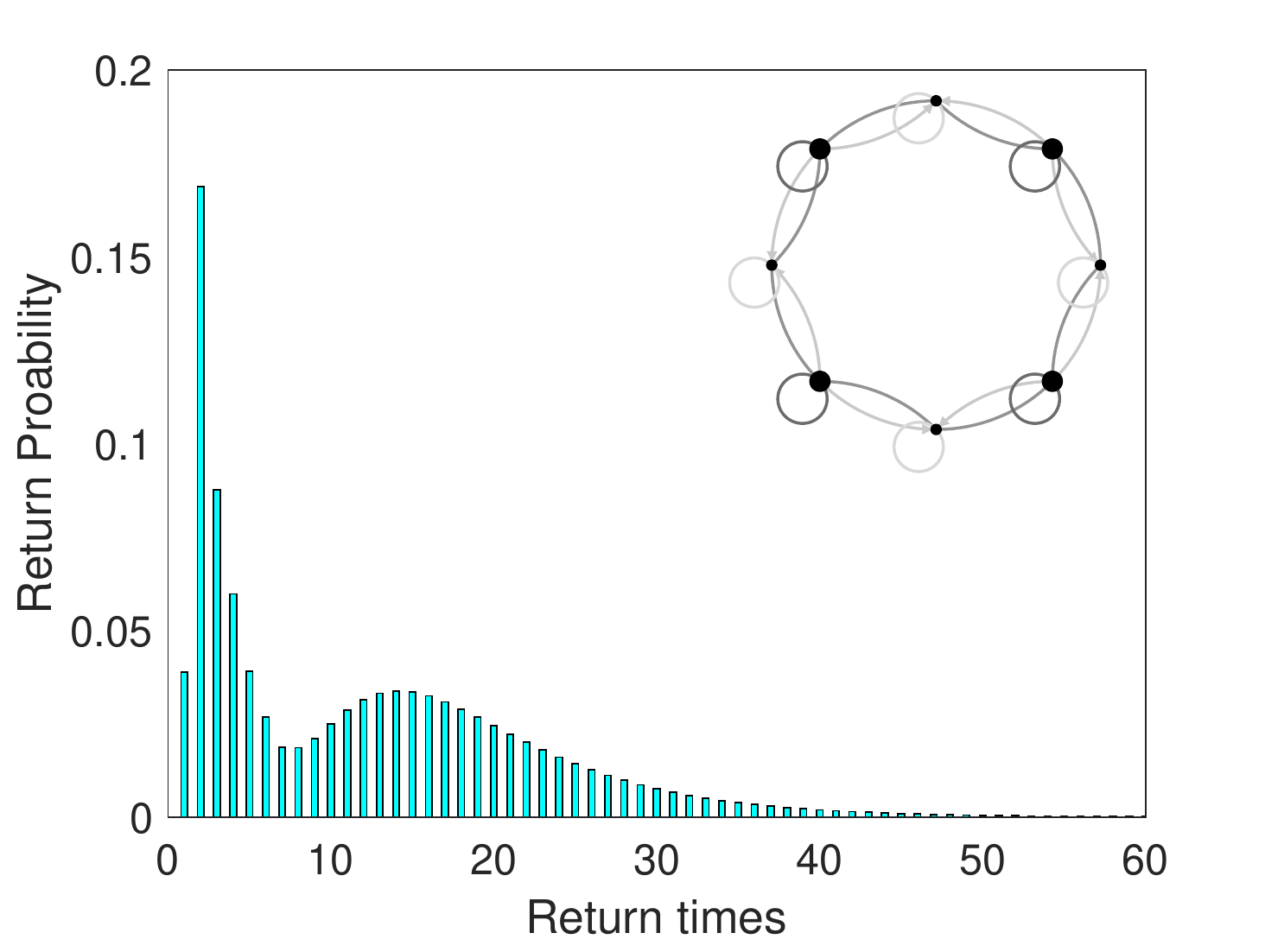}}
\subfigure[ MaxEntropyRate chain on ring graph]{
\includegraphics[scale=0.38]{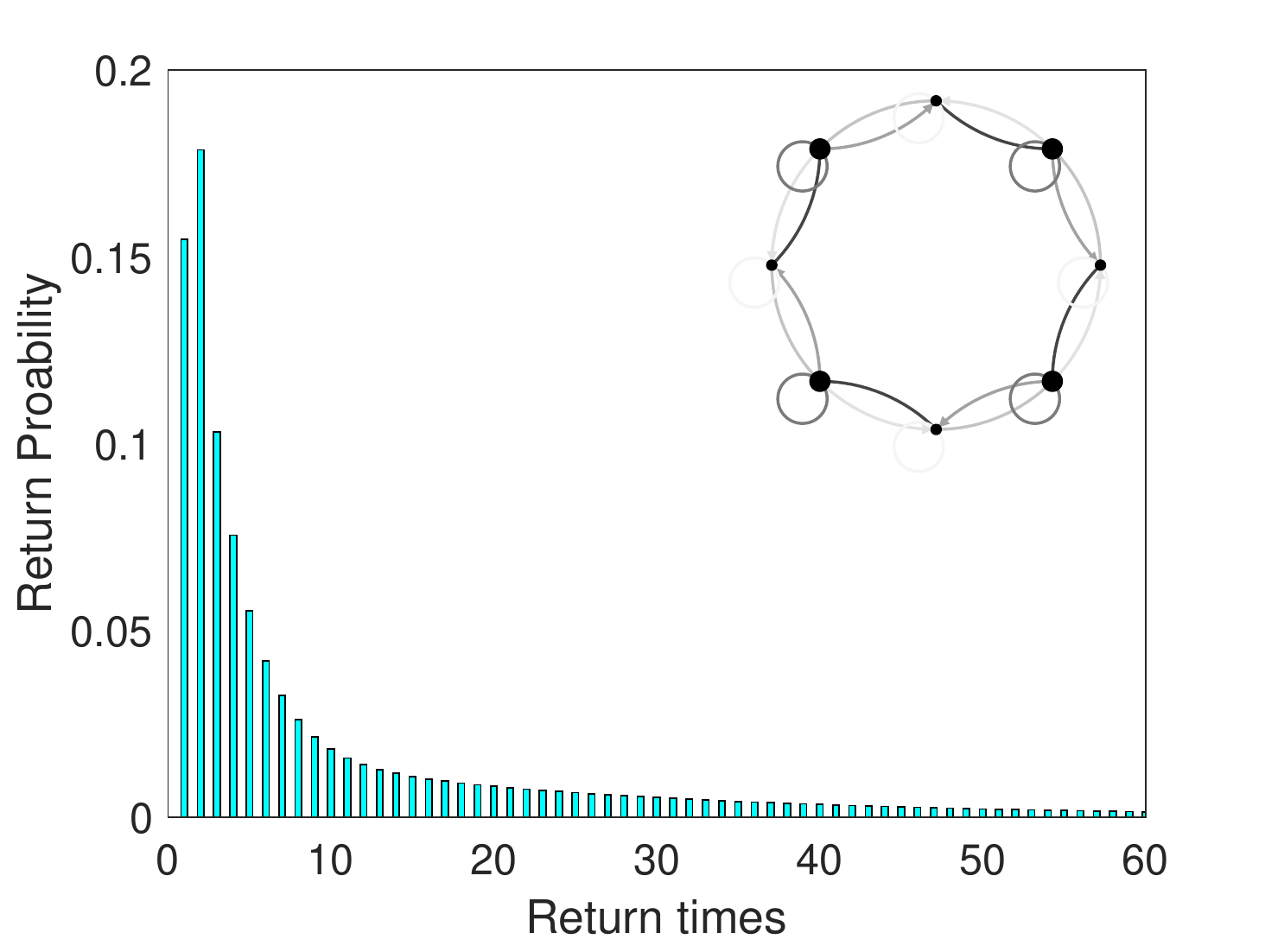}}
\subfigure[MinKemeny chain on ring graph]{
\includegraphics[scale=0.38]{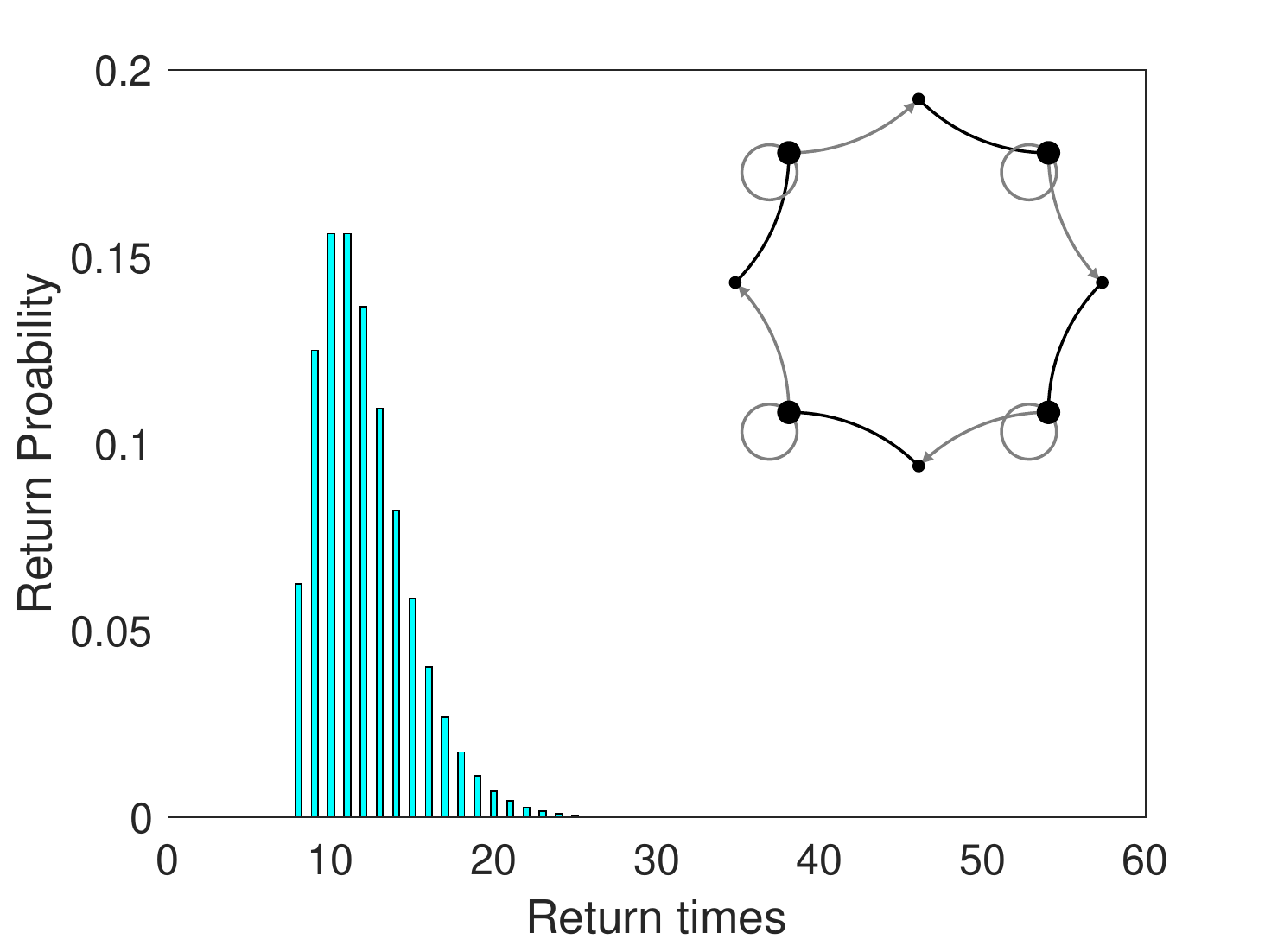}}
\caption{Return time distributions of node $1$ (i.e., top node) on an
  $8$-node ring graph with stationary distribution
  $\bm{\pi}=[1/12,1/6,\dots,1/12,1/6]^\top$. Although the expectations of
  the first return time distributions in the figure are the same, the
  histogram is remarkably different for different chains. Specifically, for
  the nonreversible MaxRetrunEntropy chain, the distribution is bimodal and
  generates more entropy. The node size is proportional to the stationary distribution.}\label{Fig:returndistribution}
\end{center}
\end{figure*}

\begin{figure*}
\begin{center}
\subfigure[MaxReturnEntropy chain on $4\times4$ grid]{
\includegraphics[scale=0.38]{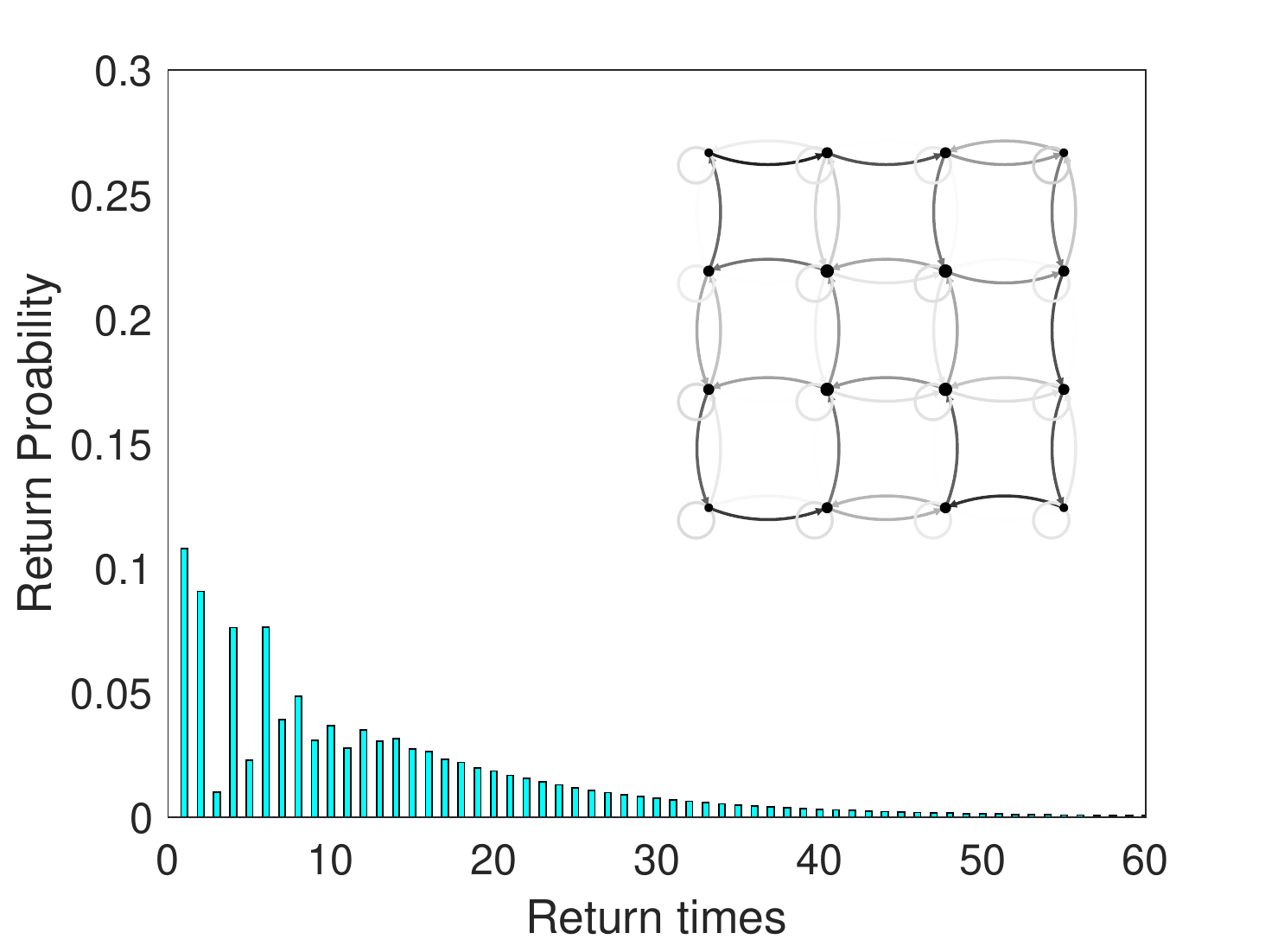}}
\subfigure[ MaxEntropyRate chain on $4\times4$ grid]{
\includegraphics[scale=0.38]{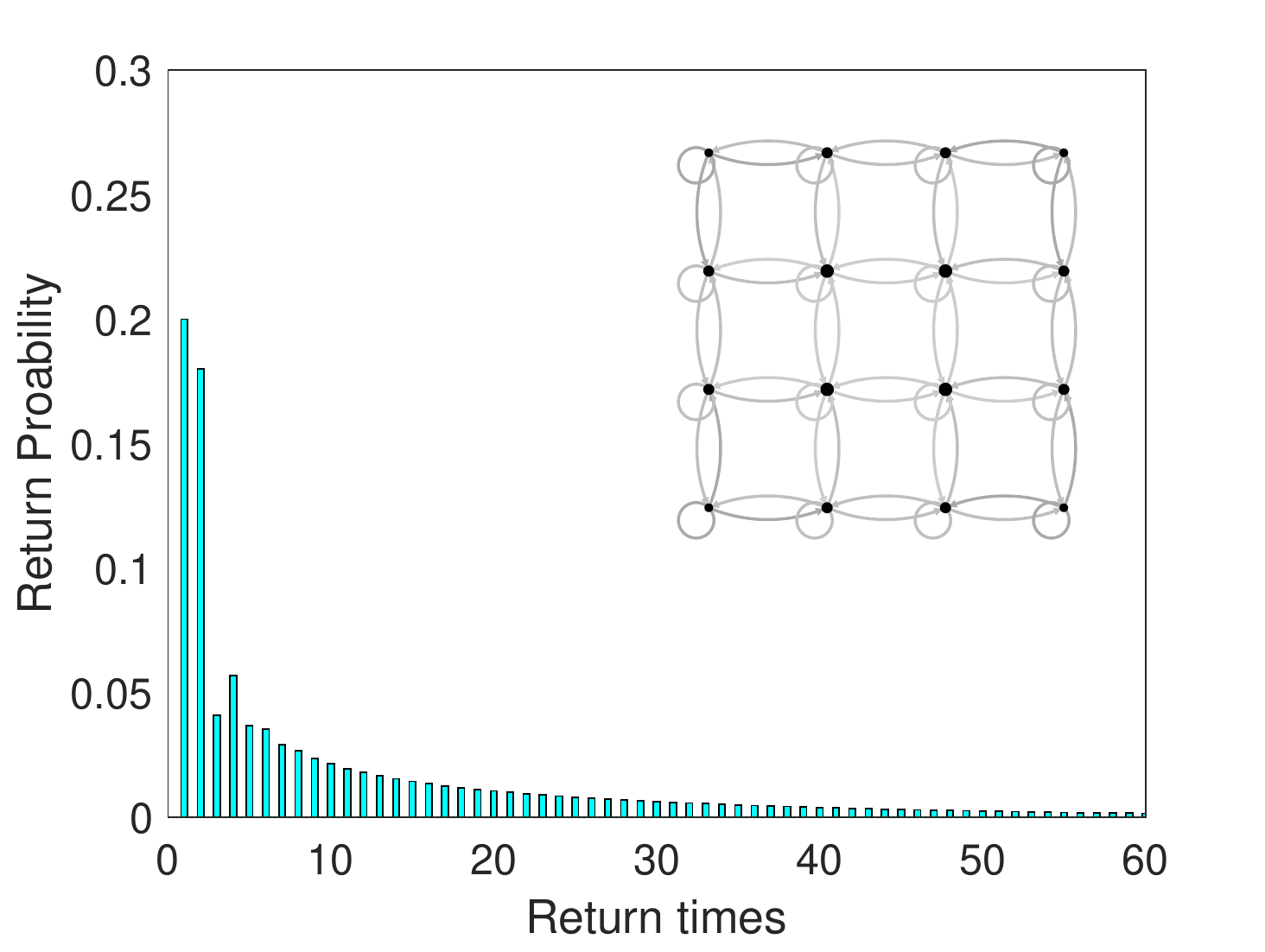}}
\subfigure[MinKemeny chain on $4\times4$ grid]{
\includegraphics[scale=0.38]{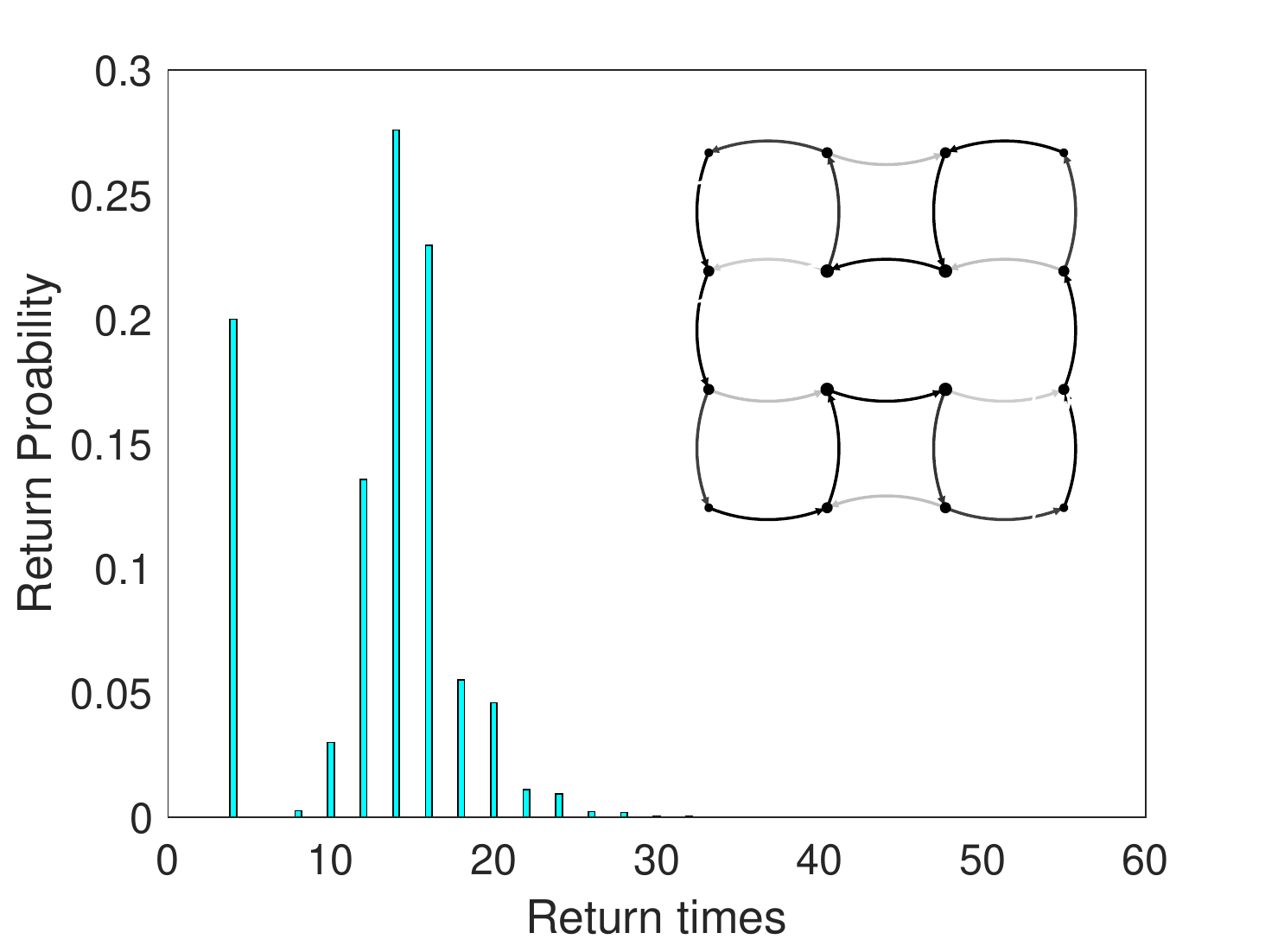}}
\caption{Return time distributions of node $6$ (i.e., second node on the second row) on a $4\times4$ grid with stationary distribution $\bm{\pi}$ proportional to the node degree and unitary travel times. The node size is proportional to the stationary distribution.}\label{Fig:returndistribution_grid}
\end{center}
\end{figure*}

\begin{flushleft}
{\bf{Chains on graphs with unitary travel times}}
\end{flushleft}

\emph{Comparison}: We consider $2$ simple undirected graphs and solve for the MaxReturnEntropy chain, the MaxEntropyRate chain and the MinKemeny chain for each case. We compare the return time entropy, the entropy rate, and the Kemeny constant of these chains in Table~\ref{comptable}. The stationary distribution of the ring graph is set to be $\bm{\pi}=[1/12,1/6,\dots,1/12,1/6]^\top$, and the stationary distribution of of grid is proportional to the degree of nodes. To evaluate the value of $\Hrt(P)$, we set $\eta=10^{-2}$. From the table, we notice that the MaxReturnEntropy chain has the highest value of the return time entropy in both cases. It also has relatively good performance in terms of the entropy rate and the Kemeny constant, which indicates that the MaxReturnEntropy chain is potentially a good combination of speed (expected traversal time) and unpredictability. Furthermore, it is clear that \eqref{eq:relations}, which characterizes the relationship between the entropy rate and the return time entropy, holds.

\begin{table}[h]
\centering
\caption{Comparison between different chains on different graphs}\label{comptable}
\begin{tabular}{ |c|c|c|c|c| }
\hline
Graph&Markov chains &$\Hrt(P)$&$\Hrate(P)$ & \makecell{Kemeny\\ constant} \\ \hline
\multirow{3}{*}{$8$-node ring}
&MaxReturnEntropy   &2.4927&0.8698&10.0479 \\
&MaxEntropyRate  &2.3510&0.9883&19.5339 \\
&MinKemeny&1.9641 &0.4621 &6.1667\\ \hline
\multirow{3}{*}{$4$-by-$4$ grid}
&MaxReturnEntropy   &3.6539&0.9491&16.3547 \\
&MaxEntropyRate  &3.2844&1.4021&30.8661 \\
&MinKemeny&2.0990 &0.2188 &10.0938\\ \hline
\end{tabular}
\end{table}

\emph{Observations}: In computing the MaxReturnEntropy chain, we observe some interesting properties of our problem. First, when solving Problem \ref{probLP} by the gradient projection method with different initial conditions, we found different optimal solutions, and they have slightly different optimal values. This suggests that Problem \ref{prob:weightedentropy} is unlikely to be a convex problem. Secondly, the global optimal solution to Problem \ref{prob:weightedentropy} is possibly not unique in general. For instance, for an undirected ring graph with even number of nodes and certain stationary distribution, exchanging the probability of going right and that of going left for all nodes does not change the return time entropy. Thirdly, the optimal solution to Problem \ref{prob:weightedentropy} is likely to be nonreversible because none of the approximate optimal solutions we have encountered are reversible. This again indicates that the MaxReturnEntropy chain is a good combination of unpredictability and speed. Fourth, even if we set the edge weight $\epsilon=0$, the MaxReturnEntropy chain is always irreducible.



\emph{Intuitions}: In order to provide intuitions for the maximization of the return time entropy, we compare and plot the chains as well as the return time distribution of a same node on the $8$-node ring graph and the $4\times4$ grid graph in Fig. \ref{Fig:returndistribution} and Fig. \ref{Fig:returndistribution_grid}, respectively. Since the stationary distribution is fixed and identical for all chains in each case, the expectations of the probability mass functions in each figure are the same. From the figures, we note that for the MaxReturnEntropy chain, the return time distribution is reshaped so that the distribution is more spread out and it is more difficult to predict the return time. In contrast, the return time distribution for the MinKemeny chain has a predictable pattern and the return time probability is constantly $0$ for some time intervals. Moreover, from the visualization of the chains, we notice that the MaxReturnEntropy chain has a net flow on the graph, which again indicates its nonreversibility.

\begin{flushleft}
{\bf{MaxReturnEntropy and MinKemeny on a realistic map}}
\end{flushleft}

In this part, we compare the MaxReturnEntropy chain with the MinKemeny
chain on a realistic map with travel times. The problem data is taken from
\cite[Section 6.2]{SA-EF-SLS:14}: a small area in San Francisco (SF) is
modeled by a fully connected directed graph with $12$ nodes and by-car
travel times on edges measured in seconds. The map is shown in
Fig. \ref{fig:SF_location_crime_map}. The importance of the a location
(node) is characterized by the the number of crimes recorded at that place
during a specific period, and the surveillance agent should visit the
places with higher crime rate more often. The visit frequency is set to be
$ [\frac{133}{866}, \frac{90}{866},\frac{89}{866},\frac{87}{866},
  \frac{83}{866}, \frac{83}{866},\frac{74}{866}, \frac{64}{866},
  \frac{48}{866}, \frac{43}{866}, \frac{38}{866}, \frac{34}{866}]^\top $.
For simplicity, we quantize the travel times by treating a minute as one unit of time, i.e., dividing the travel times by $60$ and round the result to the smallest integer that is larger than it, and by doing so, we have $w_{\max}=9$. The pairwise travel times are recorded in Table \ref{tb:pairwiseTime}.


\begin{figure}[htbp]
\begin{center}
 \includegraphics[width=0.8\linewidth,height=.15\textheight]{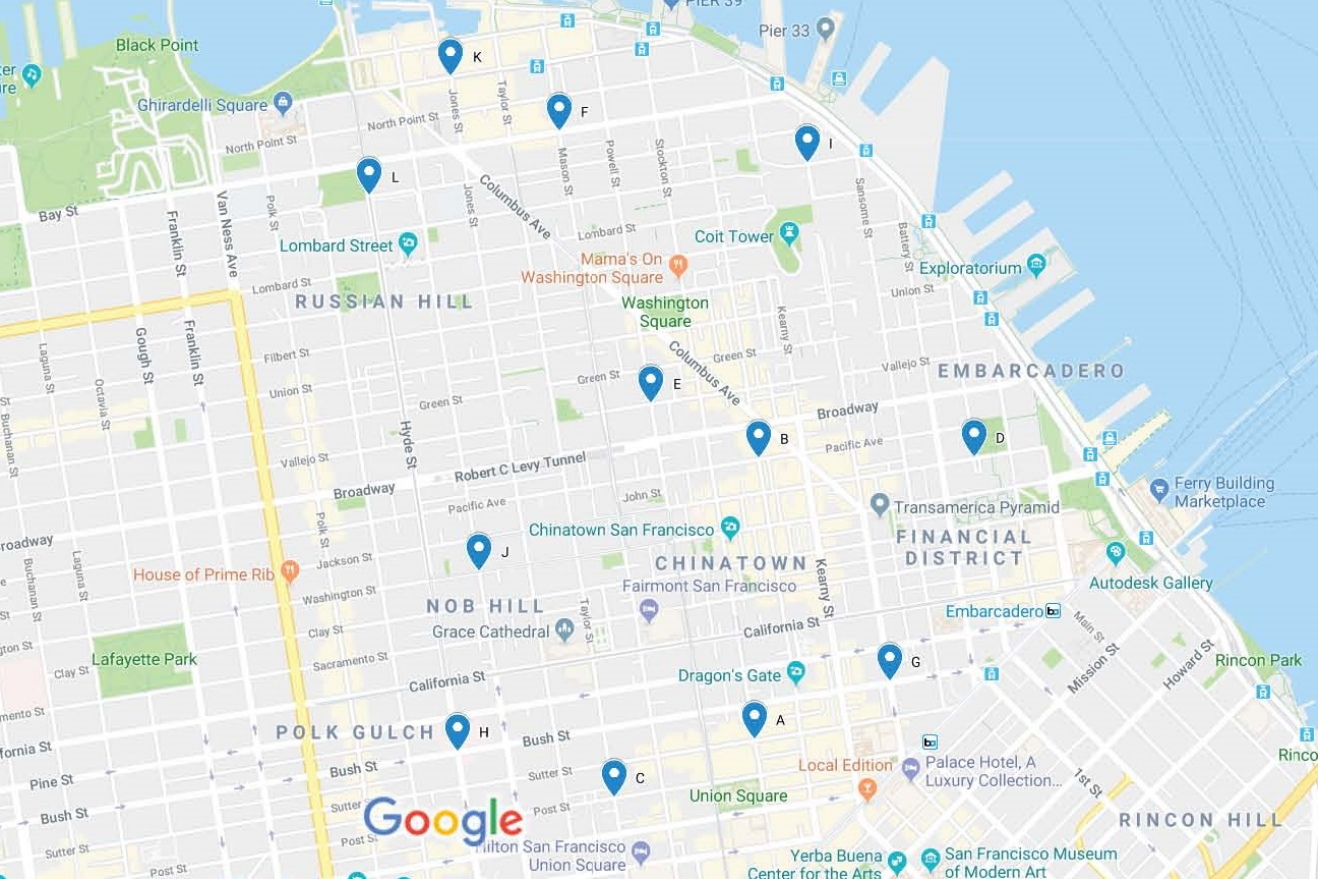}
\caption{San Francisco (SF) crime map from \cite[Section 6.2]{SA-EF-SLS:14}.}\label{fig:SF_location_crime_map}
\end{center}
\end{figure}

\begin{table}[h]
\centering
\caption{The quantized pairwise by-car travel times on SF crime map}\label{tb:pairwiseTime}
\begin{tabular}{ ccccccccccccc }
\hline\rule{0pt}{1\normalbaselineskip}
Location&A&B&C&D&E&F&G&H&I&J&K&L \\ \hline\rule{0pt}{1\normalbaselineskip}
A&1&	3&	3&	5&	4&	6&	3&	5&	7&	4&	6&	6\\
B&3&	1&	5&	4&	2&	4&	4&	5&	5&	3&	5&	5\\
C&3&	5&	1&	7&	6&	8&	3&	4&	9&	4&	8&	7\\
D&6&	4&	7&	1&	5&	6&	4&	7&	5&	6&	6&	7\\
E&4&	3&	6&	5&	1&	3&	5&	5&	6&	3&	4&	4\\
F&6&	4&	8&	5&	3&	1&	6&	7&	3&	6&	2&	3\\
G&2&	5&	3&	5&	6&	7&	1&	5&	7&	5&	7&	8\\
H&3&	5&	2&	7&	6&	7&	3&	1&	9&	3&	7&	5\\
I&8&	6&	9&	4&	6&	4&	6&	9&	1&	8&	5&	7\\
J&4&	3&	4&	6&	3&	5&	5&	3&	7&	1&	5&	3\\
K&6&	4&	8&	6&	4&	2&	6&	6&	4&	5&	1&	3\\
L&6&	4&	6&	6&	3&	3&	6&	4&	5&	3&	2&	1  \\\hline
\end{tabular}
\end{table}

First, we compare three key metrics of the MaxReturnEntropy chain and MinKemeny chain. The results are reported in Table \ref{tb:SFcompare}. It can be observed that the MaxReturnEntropy chain is much better than the MinKemeny chain regarding the return time entropy and the entropy rate. This better performance in terms of the unpredictability is obtained at the cost of being slower as indicated by the larger weighted Kemeny constant.

\begin{table}[h]
\centering
\caption{Comparison between different chains on SF crime map}\label{tb:SFcompare}
\begin{tabular}{ |c|c|c|c| }
\hline
Markov chains &$\Hrt(P)$&$\Hrate(P)$ & \makecell{Weighted Kemeny\\ constant} \\ \hline
MaxReturnEntropy &5.0078&1.7810&63.6007\\\hline
MinKemeny&2.4678&0.6408&24.2824\\ \hline
\end{tabular}
\end{table}

We also plot the return time distribution of location A in Fig. \ref{fig:SF_returndistr}. Apparently, the MaxReturnEntropy chain spreads the return time probabilities over the possible return times and it is hard to predict the exact time the surveillance agent comes back to the location. In contrast, the MinKemeny chain tries to achieve fast traversal on the graph and the return times distribute over a few intervals.

\begin{figure}[htbp]
\begin{center}
\subfigure[ MaxReturnEntropy chain on SF crime map]{
\includegraphics[scale=0.5]{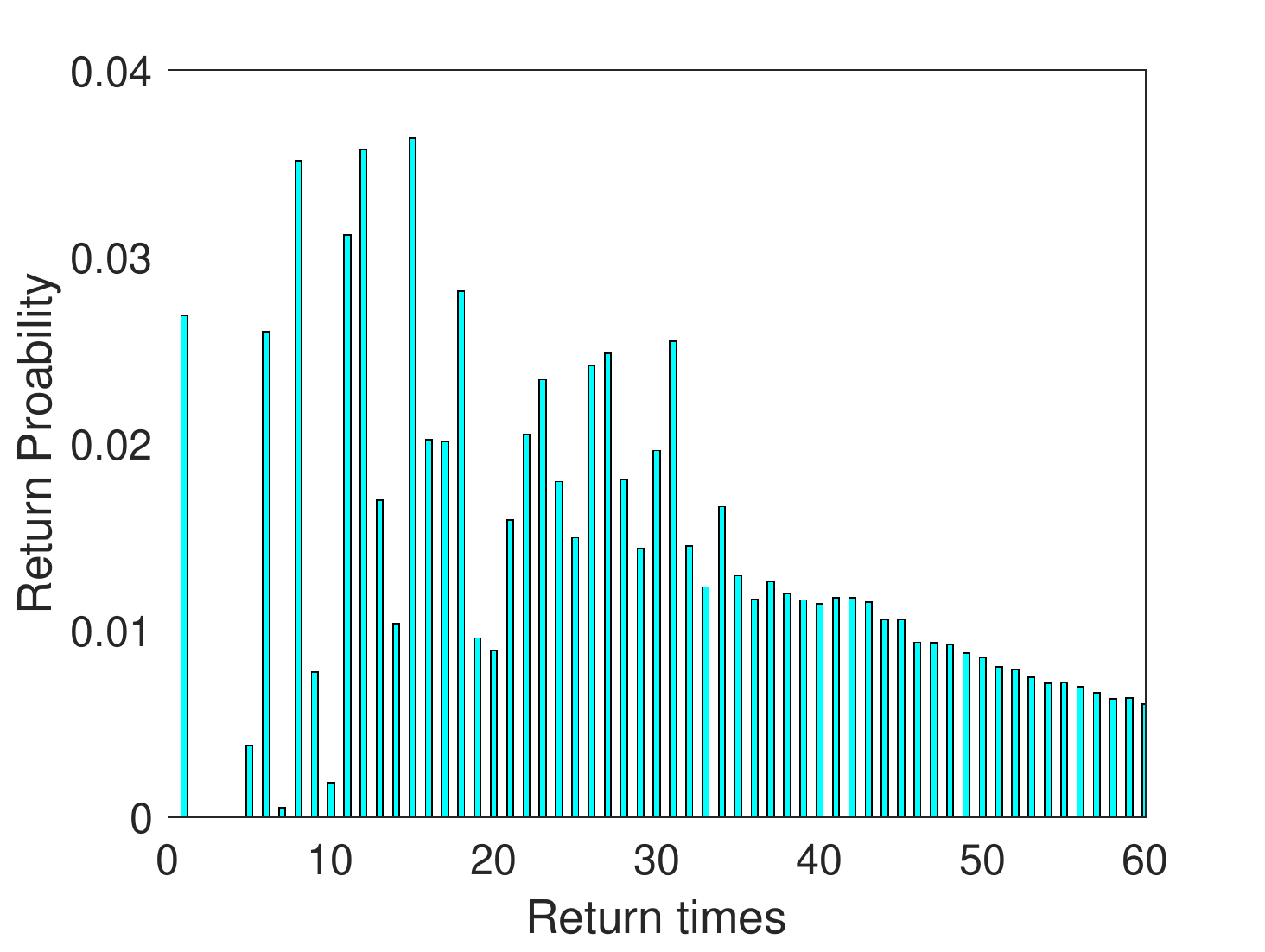}}
\subfigure[ MinKemeny chain on SF crime map]{
\includegraphics[scale=0.5]{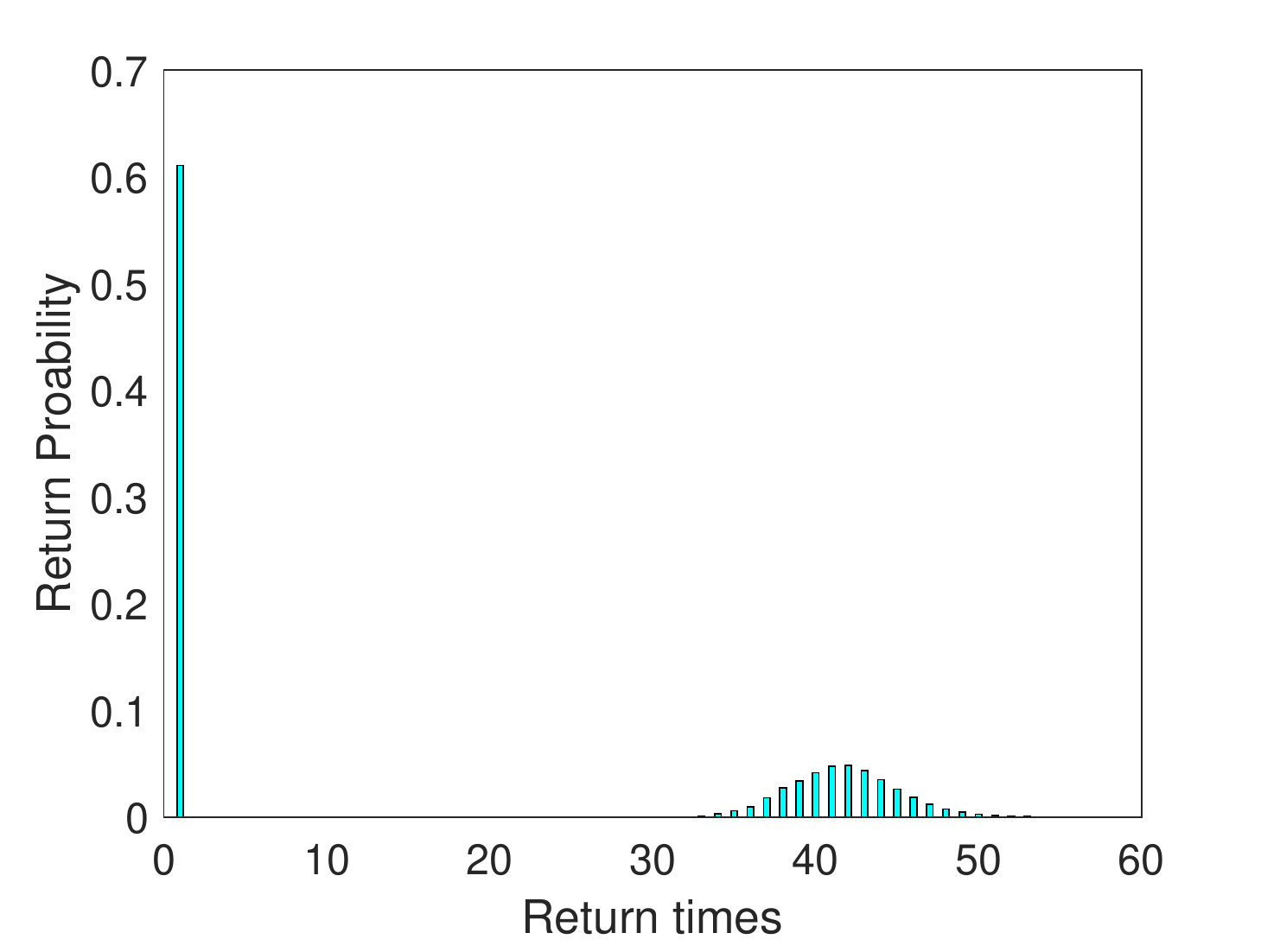}}
\caption{Return time distributions of location A on SF crime map. Note that the scales of the vertical axes are different in the two figures.}\label{fig:SF_returndistr}
\end{center}
\end{figure}


\subsection{Application to the Robotic Surveillance Problem}\label{subsec:app-robotics}
In this subsection, we provide simulation results in the application
of robotic surveillance.

\emph{Setup}: Consider the scenario where a single agent performs the
surveillance task by moving randomly according to a Markov chain on the
road map. The intruder is able to observe the local behaviors of the
surveillance agent, e.g., presence/absence and duration between visits, and
he/she plans and decides the time of attack so as to avoid being
captured. It takes a certain amount of time for the intruder to complete an
attack, which is called the \emph{attack duration} of the intruder. A
successful detection/capture happens when the surveillance agent and the
intruder are at the same location and the intruder is attacking.

\emph{Intruder model (success probability maximizer with bounded
  patience)}: Consider a rational intruder that exploits the return time
statistics of the Markov chains and chooses an optimal attack time so as to
minimize the probability of being captured. The intruder picks a node $i$
to attack randomly according to the stationary distribution, and it
collects and learns the probability distribution of node $i$'s first return
time. Suppose the intruder and the surveillance agent are at the same node
$i$ at the beginning and the attack duration of the intruder is $\tau$. If
the intruder observes that the surveillance agent leaves the node and does
not come back for $s$ periods, he/she can attack with the probability of
being captured given by
\begin{equation}\label{eq:condcumureturn}
\sum_{k=1}^{\tau}\mathbb{P}(T_{ii}=s+k\,|\,T_{ii}> s).
\end{equation}
Mathematically speaking, \eqref{eq:condcumureturn} is the conditional cumulative return probability for the surveillance agent. Specifically for $s=0$, \eqref{eq:condcumureturn} is the capture probability when the intruder attacks immediately after the agent leaves the node. Then, the optimal time of attack $s_i$ for the intruder is given by
\begin{equation}\label{eq:attackinstant}
s_i=\argmin_{0\leq s\leq S_i}\{\sum_{k=1}^{\tau}\mathbb{P}(T_{ii}=s+k\,|\,T_{ii}> s)\}.
\end{equation}
The reason there is an upper bound $S_i$ on $s$ is that the event
$T_{ii}> s$ happens with very low probability when $s$ is large,
and the intruder may be unwilling to wait for such an event to happen. Let $\delta\in(0,1)$ be the \emph{degree of impatience} of the intruder, then $S_i$ can be chosen as the minimal positive integer such that the following holds,
\begin{equation*}
\mathbb{P}(T_{ii}\geq S_i) \leq \delta,
\end{equation*}
where a smaller $\delta$ implies a larger $S_i$ and a more patient intruder. In other words, when $\delta$ is small, the intruder is willing to wait for a rare event to happen.
Note that the value of $S_i$ is also dependent on the node $i$ that the intruder chooses to attack, and thus the argmin in \eqref{eq:attackinstant} is over different ranges when the intruder attacks different nodes. In summary, the intruder is dictated by two parameters: the attack duration $\tau$ and the degree of impatience $\delta$, and the strategy for the intruder is as follows: waits until the event
that the surveillance agent leaves and does not come back for the
first $s_i$ steps happens, then attacks immediately.

From the surveillance point of view, the probability of capturing the rational intruder when he/she attacks node $i$ is
\begin{equation*}
\mathbb{P}_i(\text{Capture}) = \sum_{k=1}^{\tau}\mathbb{P}(T_{ii}=s_i+k\,|\,T_{ii}> s_i),
\end{equation*}
and the performance of the Markov chains can be evaluated by the total probability of capture as follows
\begin{equation}\label{eq:prob_capture}
\mathbb{P}(\text{Capture}) = \sum_{i=1}^{n}\pi_i\mathbb{P}_i(\text{Capture}).
\end{equation}



\emph{Simulation results}: Designing an optimal defense mechanism for the rational intruder is an interesting yet challenging problem in its own. Instead, we use the MaxReturnEntropy chain as a heuristic solution and compare its performance with other chains. In the following, we consider two types of graphs: the grid graph and the SF crime map. The degree of impatience of the intruder is set to be $\eta = 0.1$ in this part.

We first consider a $4\times4$ grid and plot the probability of
capture defined by \eqref{eq:prob_capture} for the chains in comparison in Fig. \ref{fig:comparison_grid}. It can be observed that, when defending against the rational intruder described above, the MaxReturnEntropy chain outperforms all other
chains when the attack duration of the intruder is small or moderate. The unpredictability in the return time prevents the rational intruder from taking advantage of the visit statistics learned from the observations. The MinKemeny chain, which emphasizes a faster traversal, has a hard
time capturing the intruder when the attack duration of the intruder is small. This is because the agent moves in a
relatively more predictable way, and the return time statistics may have a pattern that could be exploited. The MaxEntropyRate chain has the in-between performance. 


\begin{figure}[htbp]
\begin{center}
\includegraphics[scale=0.5]{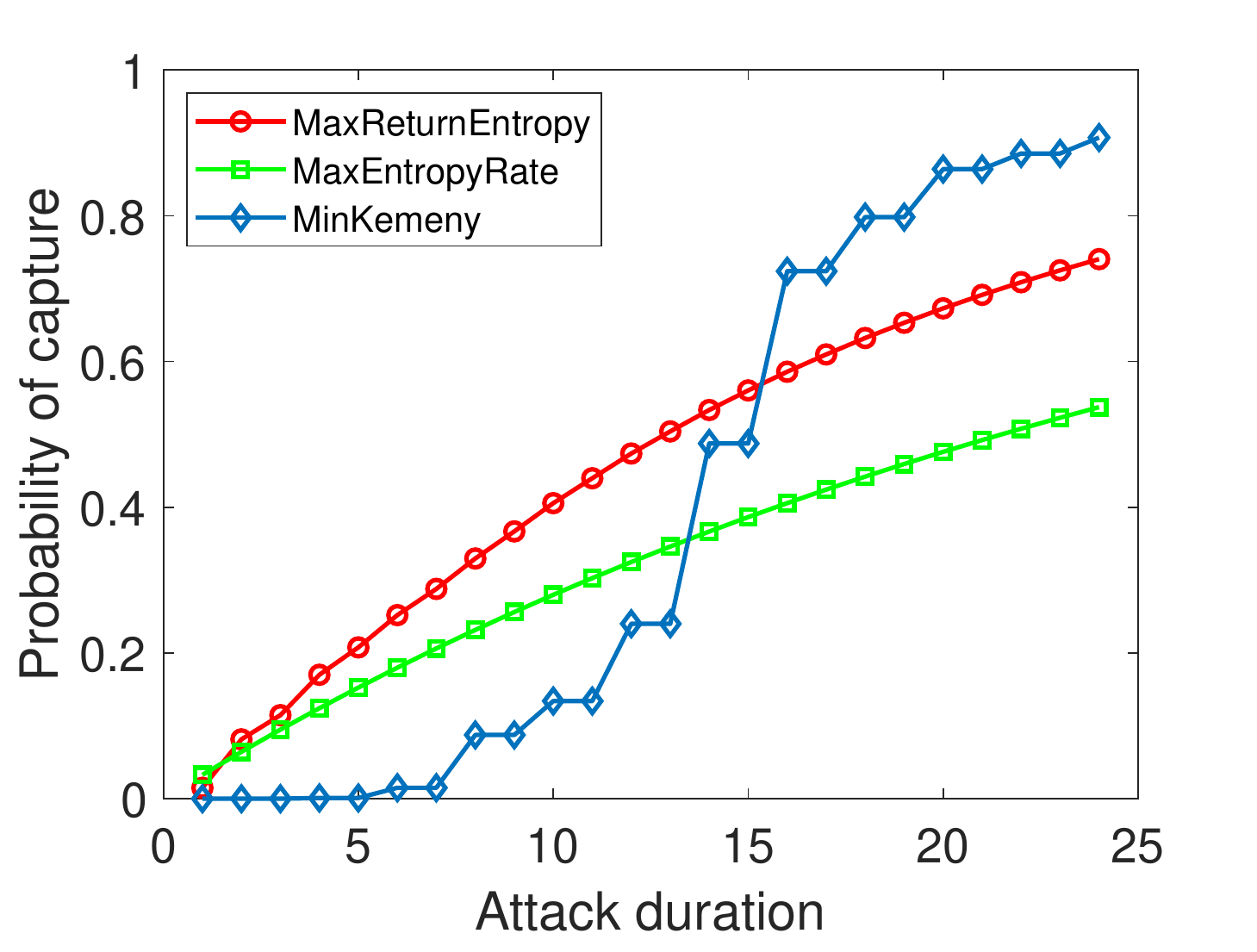}
\caption{Performance of different chains on a $4\times4$ grid.}\label{fig:comparison_grid}
\end{center}
\end{figure}


%
%

For the SF crime map, we use the same problem data as described in Subsection \ref{subsec:comp-MRE}. Since the MaxEntropyRate chain does not generalize to the case when there are travel times, we compare the performance of the MaxReturnEntropy chain and the MinKemny chain. Again, The MaxReturnEntropy chain outperforms the MinKemeny chain when the attack duration of the intruder is relatively small.

\begin{figure}[htbp]
\begin{center}
\includegraphics[scale=0.5]{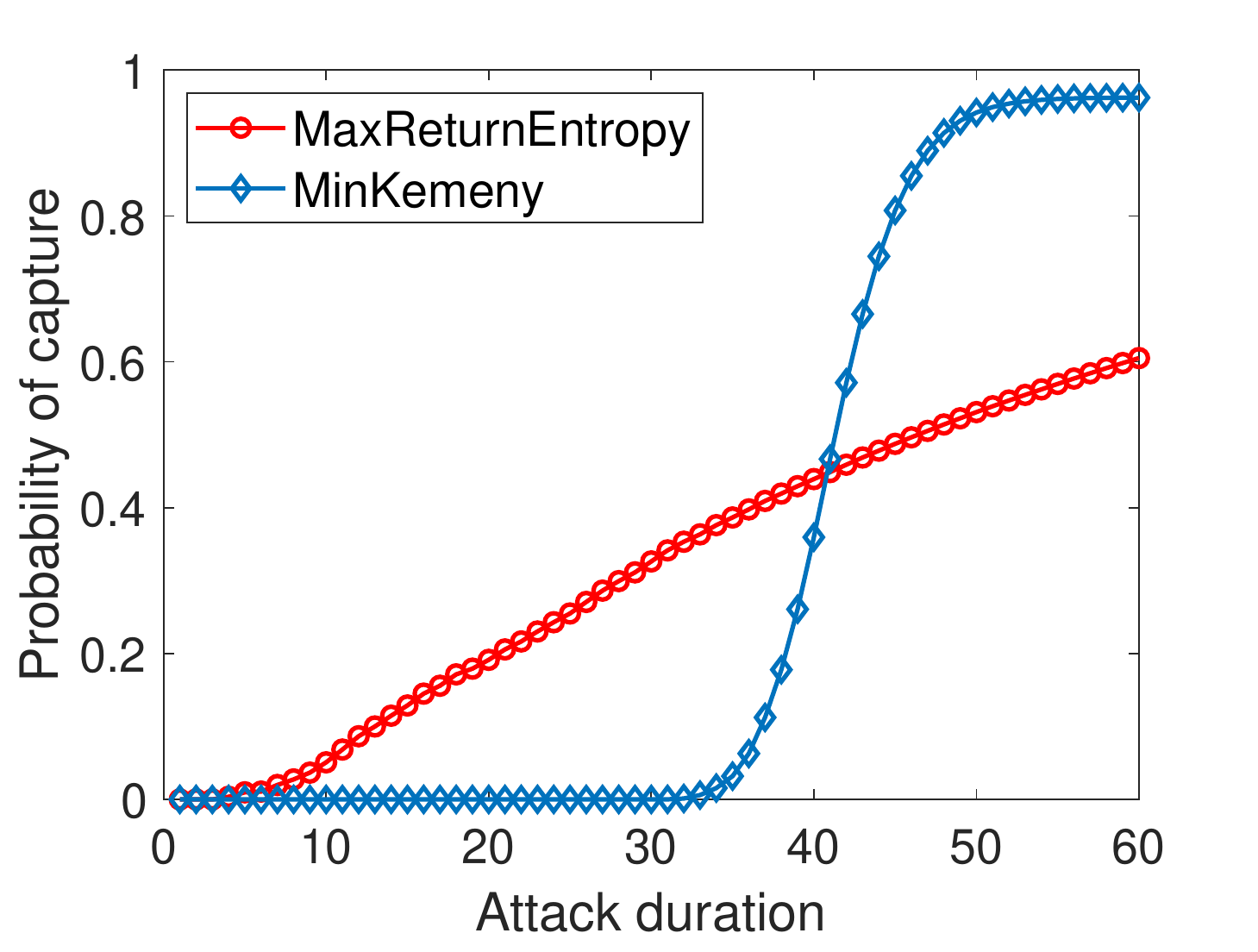}
\caption{Performance of different chains on the SF crime map.}\label{fig:SF_crime_map}
\end{center}
\end{figure}

\emph{Summary}: The simulation results presented in this subsection demonstrate that the MaxReturnEntropy chain is an effective strategy against the intruder with reasonable amount of knowledge and level of intelligence, particularly when the attack duration of the intruder is small or moderate. With the property of both unpredictability and speed, the MaxReturnEntropy chain should also work well in a much more broader range of scenarios.

\section{Conclusion}\label{sec:Conclusion}
In this paper, we proposed and optimized a new metric that quantifies the unpredictability of Markov chains over a directed strongly connected graph with travel times, i.e., the return time entropy. We characterized the return time probabilities and showed that optimizing the return time entropy is a well-posed problem. For the case of unitary travel times, we established an upper bound for the return time entropy by using the maximum entropy principle and obtained an analytic solution for the complete graph. We connected the return time entropy with the well-known entropy rate of Markov chains and showed that the return time entropy is lower bounded by the entropy rate and upper bounded by $n$ times the entropy rate. In order to solve the optimization problem numerically, we approximated the return time entropy as well as a practically useful conditional return time entropy by the truncated return time entropy. We derived the gradient of the truncated return time entropy and proposed to solve the problem by the gradient projection method. We applied our results to the robotic surveillance problem and found that the chain with maximum return time entropy is a good trade-off between speed and unpredictability, and it performs better than several existing chains against a rational intruder.

A number of problems are still open. First of all, a simple closed-form expression for the return time entropy would enable us to establish more properties of the objective function and thus make the optimization problem more tractable. Second, it is interesting to design a best Markov chain directly that defends against the intruder model proposed in this paper. Third, how to generalize the results to the case of multiple robots remains to be investigated. Fourth, we believe there are more application scenarios for Markov chains where the return time entropy is an appropriate quantity to optimize.

\bibliographystyle{plainurl}
\bibliography{alias,Main,FB}

\end{document}